\documentclass{siamart190516}

\usepackage{amssymb}
\usepackage{amsmath}
\usepackage{bm}             
\usepackage[mathscr]{eucal}
\usepackage{color}
\usepackage{bbm}

\usepackage{lipsum}
\usepackage{amsfonts}
\usepackage{graphicx}
\usepackage{epstopdf}
\usepackage{algorithmic}
\ifpdf
  \DeclareGraphicsExtensions{.eps,.pdf,.png,.jpg}
\else
  \DeclareGraphicsExtensions{.eps}
\fi

\newsiamremark{remark}{Remark}
\newsiamremark{hypothesis}{Hypothesis}
\crefname{hypothesis}{Hypothesis}{Hypotheses}
\newsiamthm{claim}{Claim}

\headers{Spherically-Symmetric Plasmas}{S. Pankavich}

\title{Exact Large Time Behavior of Spherically-Symmetric Plasmas\thanks{Submitted to the editors DATE.
\funding{The author was supported in part by NSF grants DMS-1614586 and DMS-1911145.}}}

\author{Stephen Pankavich\thanks{Department of Applied Mathematics and Statistics, Colorado School of Mines, Golden, CO 
  (\email{pankavic@mines.edu).}}}

\usepackage{amsopn}

\newcommand{\Omw}{\Om_w}
\newcommand{\Oml}{\Om_\ell}
\newcommand{\Om}{\Omega}
\newcommand{\winf}{w}
\newcommand{\linf}{\ell}
\newcommand{\chfn}{\mathbbm{1}}
\newcommand{\bfR}{\mathbb{R}}
\newcommand{\mcM}{\mathcal{M}}
\newcommand{\mcX}{\mathcal{X}}
\newcommand{\mcV}{\mathcal{V}}
\newcommand{\mcR}{\mathcal{R}}
\newcommand{\mcU}{\mathcal{U}}
\newcommand{\bS}{S}
\newcommand{\bbS}{\overline{S}}
\newcommand{\mfR}{\mathfrak{R}}
\newcommand{\mfW}{\mathfrak{W}}
\newcommand{\mfF}{\mathfrak{F}}
\newcommand{\mfV}{\mathfrak{V}}
\newcommand{\mcD}{\mathcal{D}}
\newcommand{\mcW}{\mathcal{W}}
\newcommand{\mcL}{\mathcal{L}}
\newcommand{\mcB}{\mathcal{B}}
\newcommand{\finf}{F_\infty}
\newcommand{\mcJ}{\mathcal{J}}
\newcommand{\mcEVP}{\mathcal{E}_\mathrm{VP}}
\newcommand{\mcERVP}{\mathcal{E}_\mathrm{RVP}}

\newcommand{\Tm}{T_0}
\newcommand{\Rm}{\mcR_0}
\newcommand{\Vm}{\mcV_0}

\begin{document}

\maketitle

\begin{abstract}
We consider the classical and relativistic Vlasov-Poisson systems with spherically-symmetric initial data and prove the optimal decay rates for all suitable $L^p$ norms of the charge density and electric field, as well as, the optimal growth rates for the largest particle position and momentum on the support of the distribution function. Though a previous work \cite{Horst} established upper bounds on the decay of the supremum of the charge density and electric field, we provide a slightly different proof, attain optimal rates, and extend this result to include all other norms. Additionally, we prove sharp lower bounds on each of the aforementioned quantities and establish the time-asymptotic behavior of all spatial and momentum characteristics.  Finally, we investigate the limiting behavior of the spatial average of the particle distribution as $t \to \infty$. In particular, we show that it converges uniformly to a smooth, compactly-supported function that preserves the mass, angular momentum, and energy of the system and depends only upon limiting particle momenta. 
\end{abstract}



\textbf{Dedication:} Dedicated to the memory of Bob Glassey, a friend and mentor.

\section{Introduction}
\label{intro}

The motion of a collisionless plasma with a single species of charge is given by the three-dimensional, relativistic Vlasov-Maxwell system:
\begin{equation}
\label{RVM}
\left \{ \begin{gathered}
\partial_{t}f+\hat{v} \cdot\nabla_{x}f+(E + \hat{v} \times B) \cdot\nabla_{v}f=0\\
\partial_{t} E=\nabla \times B- 4\pi j, \quad \nabla \cdot E=4\pi\rho,\\
\partial_{t} B=-\nabla \times E, \quad\nabla \cdot B=0
\end{gathered} \right .
\end{equation}
where
$$\rho(t,x)=\int_{\mathbb{R}^3} f(t,x,v)\,dv, \quad j(t,x)=\int_{\mathbb{R}^3} \hat{v} f(t,x,v)\,dv$$
and
$\hat{v} = \frac{v}{\sqrt{1 + \vert v \vert^2}}$.
Here, $t \geq 0$ represents time while $x,v \in \bfR^3$ are position and momentum, $\hat{v}$ is the relativistic velocity, $f(t,x,v)$ is the particle distribution function, $\rho(t,x)$ is the associated charge density, $E(t,x)$ and $B(t,x)$ are the self-consistent electric and magnetic fields generated by the charged particles, and we have chosen units such that the mass and charge of each particle are normalized.
In the classical limit (i.e., as the speed of light $c \to \infty$) this system reduces \cite{Degond, Jacklimit} to the Vlasov-Poisson system:
 \begin{equation}
\label{VPgeneral}
\left \{ \begin{aligned}
& \partial_{t}f+v\cdot\nabla_{x}f+E \cdot\nabla_{v}f=0\\
& \rho(t,x)=\int_{\mathbb{R}^3} f(t,x,v)\,dv\\
& E(t,x) = \int_{\mathbb{R}^3} \frac{x-y}{\vert x - y \vert^3} \rho(t,y) \ dy
\end{aligned} \right .
\end{equation}
with the initial condition $f(0,x,v) = f_0(x,v)$.

In the present paper, we consider the Cauchy problem for \eqref{VPgeneral} and require initial data $f_0 \in C^1_c(\mathbb{R}^6)$ that is spherically-symmetric and nonnegative (as in \cite{Horst}).
This symmetry assumption leads to a reduction in the system of PDEs that we will describe later as \eqref{VP}.
As an alternative to studying the classical system, one may instead assume that the initial data provided for \eqref{RVM} is spherically-symmetric, in which case Maxwell's equations decouple and the symmetry of the solution is preserved in time, as it is for \eqref{VP}. Under this assumption the electromagnetic model reduces to the relativistic Vlasov-Poisson system with spherical-symmetry \eqref{RVP}, which we will also state later in suitable coordinates.

It is known that given smooth (and not necessarily symmetric) initial data, the Vlasov-Poisson system \eqref{VPgeneral} possesses a smooth global-in-time solution \cite{LP, Pfaf, Schaeffer}. Similarly, global classical solutions have been constructed for the relativistic Vlasov-Poisson system for spherically-symmetric initial data \cite{GS, Horst1, Horst2}.
Global existence results for these systems often depend upon precise estimates for the growth of the characteristics associated to \eqref{VPgeneral} and its relativistic analogue, which are defined by
$$\left \{
\begin{aligned}
&\dot{\mcX}(s, t, x, v)=\mcV(s, t, x, v)\\
&\dot{\mcV}(s, t, x, v)= E(s, \mcX(s, t, x, v))
\end{aligned}
\right.
\quad \mathrm{and}
\quad \left \{
\begin{aligned}
&\dot{\mcX}(s, t, x, v)=\widehat{\mcV}(s, t, x, v)\\
&\dot{\mcV}(s, t, x, v)= E(s, \mcX(s, t, x, v)),
\end{aligned}
\right.$$
respectively,
each with initial conditions
$\mcX(t, t, x, v) = x$ and
$\mcV(t, t, x, v) = v.$
%
We refer to \cite{Glassey} and \cite{Rein} as general references for these, and other, well-known models in Kinetic Theory. 

In addition to understanding the growth of characteristics, we wish to identify the exact limiting behavior of all quantities in these systems, including the maximal support of $f$, charge density, electric field, and potential energy.
In general, the Cauchy problem for such systems does not possess smooth steady states (cf., \cite{GPS5}), and thus one expects the dispersive properties of the Vlasov equation to induce the field and charge density to tend to zero as $t \to \infty$. Hence, we wish to determine $a, b, c, d \geq 0$ such that
$$\begin{gathered}
\Vert E(t) \Vert_p = \mathcal{O}\left (t^{-a}\right ), \quad \Vert \rho(t) \Vert_q  = \mathcal{O}\left (t^{-b}\right ),\\ 
\sup_{x, v} | \mcV(t,0,x,v) | \sim \mathcal{O}\left (t^{c}\right ), \quad  \sup_{x,v} |\mcX(t,0,x,v)| \sim \mathcal{O}\left (t^{d}\right )
\end{gathered}$$
for suitable $p, q \in [1,\infty]$ and $t$ sufficiently large.
We also seek to determine the limiting behavior of the spatial average of the particle distribution.

Results regarding the large time behavior of solutions to the Cauchy problem for the Vlasov-Poisson and Vlasov-Maxwell systems exist in some special cases, including small data \cite{BD}, lower-dimensional settings \cite{BKR, GPS2, Sch}, and for neutral plasmas \cite{GPS, GPS4}.
More recently, some results concerning the intermediate asymptotic behavior for these systems were recently discovered in \cite{BCP1, BCP2}.
Additionally, H\"{o}rst \cite{Horst} proved that the $L^\infty$-norm of the charge density and electric field decay for \eqref{VP} and \eqref{RVP} assuming spherically-symmetric initial data. 
In addition to obtaining sharp upper bounds on the decay rates of these quantities, we will also provide lower bounds that display the optimality of theses rates and construct a limiting distribution as $t \to \infty$ that retains some of the information lost by $f$ in the time-asymptotic limit.

Due to the assumption of radially-symmetric initial data $f_0$, it is useful to consider new variables that completely describe solutions with such symmetry.  In particular, defining the spatial radius, radial momentum, and square of the angular momentum by
\begin{equation}
\label{ang}
r = \vert x \vert, \qquad w = \frac{x \cdot v}{r}, \qquad \ell = \vert x \times v \vert^2,
\end{equation}
the radial symmetry of $f_0$ implies that the distribution function, charge density, and electric field take special forms for all time. Namely, the particle distribution $f = f(t,r,w,\ell)$ satisfies a reduced Vlasov equation, while radial representations for the charge density, enclosed mass, and electric field also result.
Therefore, we arrive at the spherically-symmetric Vlasov-Poisson system
\begin{equation}
\tag{VP}
\label{VP}
\left \{ \begin{aligned}
&\partial_{t}f+w\partial_r f+\left ( \frac{\ell}{r^3} + \frac{m(t,r)}{r^2} \right ) \partial_w f=0\\
& \rho(t,r) = \frac{\pi}{r^2} \int_0^\infty \int_{-\infty}^{\infty} f(t,r,w,\ell) \ dw \ d\ell\\
& m(t,r) = 4\pi \int_0^r \tilde{r}^2 \rho(t,\tilde{r}) \ d\tilde{r}\\
& E(t,x) = \frac{m(t,r)}{r^2} \frac{x}{r}
\end{aligned} \right .
\end{equation}
and its relativistic counterpart, the spherically-symmetric, relativistic Vlasov-Poisson system
\begin{equation}
\tag{RVP}
\label{RVP}
\left \{ \begin{aligned}
& \partial_{t}f+\frac{w}{\sqrt{1 + w^2+ \ell r^{-2}}}\partial_r f+\left ( \frac{\ell r^{-3}}{\sqrt{1 + w^2+ \ell r^{-2}}} + \frac{m(t,r)}{r^2} \right ) \partial_w f=0\\
& \rho(t,r) = \frac{\pi}{r^2} \int_0^\infty \int_{-\infty}^{\infty} f(t,r,w,\ell) \ dw \ d\ell\\
& m(t,r) = 4\pi \int_0^r \tilde{r}^2 \rho(t,\tilde{r}) \ d\tilde{r}\\
& E(t,x) = \frac{m(t,r)}{r^2} \frac{x}{r}.
\end{aligned} \right .
\end{equation}
Notice that the only difference between \eqref{VP} and \eqref{RVP} occurs in the Vlasov equation. Additionally, for both systems $| E(t,x) | = m(t,r)r^{-2}$ for every $t \geq 0$, $x \in \mathbb{R}^3$; therefore, understanding the enclosed mass will be crucial to obtaining the exact time-asymptotic behavior of the field.
Whenever necessary, we will abuse notation so as to use both Cartesian and angular coordinates to refer to functions;  for instance, the particle distribution $f$ will be written both as $f(t,x,v)$ and $f(t,r,w,\ell)$ where appropriate.
Finally, the standard conserved quantities for these systems can be easily represented in the new coordinates. In particular, the total mass, which is time-independent, can be expressed as
\begin{eqnarray*}
\mcM & = & \iint f(t, x,v) \ dv dx = \iint f_0(x,v) \ dv dx = 4\pi^2 \int_0^\infty \int_{-\infty}^{\infty} \int_0^\infty f_0(r, w, \ell) \ d\ell dw dr
\end{eqnarray*}
so that $0 \leq m(t,r) \leq \mcM$ for all $t \geq 0$ and $r > 0$.
Another conserved quantity 
%
is the energy, namely
\begin{equation}
\label{EVP}
\mcEVP 
=2\pi^2\int_0^\infty \int_{-\infty}^{\infty} \int_0^\infty (w^2 + \ell r^{-2})  f_0(r,w,\ell) \ d\ell dw dr +  2\pi \int_0^\infty \frac{m(0,r)^2}{r^2} \ dr
\end{equation}
for \eqref{VP} and 
\begin{equation}
\label{ERVP}
\mcERVP 
=4\pi^2\int_0^\infty \int_{-\infty}^{\infty} \int_0^\infty \sqrt{1+w^2 + \ell r^{-2}}  f_0(r,w,\ell) \ d\ell dw dr +  2\pi \int_0^\infty \frac{m(0,r)^2}{r^2} \ dr 
\end{equation}
for \eqref{RVP}.
Note that the compact support of $f_0$ implies $|v|\leq C$ on the support of $f_0(x,v)$, and thus
$\sqrt{w^2 + \ell r^{-2}} \leq C$ on the support of $f_0(r,w,\ell)$.
Finally, due to the structure of the spherically-symmetric Vlasov equation, any function of the angular momentum is also conserved, so that
for any $\phi \in L^1_{\mathrm{loc}}(0,\infty)$, one finds
\begin{equation}
\label{consangmom}
\mcJ_\phi = \displaystyle \int_0^\infty \int_{-\infty}^{\infty} \int_0^\infty \phi (\ell)  f(t,r,w,\ell) \ d\ell dw dr = \int_0^\infty \int_{-\infty}^{\infty} \int_0^\infty \phi (\ell)  f_0(r,w,\ell) \ d\ell dw dr
\end{equation}
for both \eqref{VP} and \eqref{RVP}.

In the angular coordinates \eqref{ang}, the characteristics of the Vlasov equation, whose notation we will often abbreviate (e.g., $\mcR(s) = \mcR(s,\tau, r, w, \ell)$),  also assume a reduced form. 
In particular, for \eqref{VP} these are
\begin{equation}
\label{charang}
\left \{
\begin{aligned}
&\dot{\mcR}(s)=\mcW(s),\\
&\dot{\mcW}(s)= \frac{\mcL(s)}{\mcR(s)^3} + \frac{m(s, \mcR(s))}{\mcR(s)^2},\\
&\dot{\mcL}(s)= 0
\end{aligned}
\right.
\end{equation}
while for \eqref{RVP} they become
\begin{equation}
\label{charangrel}
\left \{
\begin{aligned}
&\dot{\mcR}(s)=\frac{\mcW(s)}{\sqrt{1 + \mcW(s)^2 + \mcL(s) \mcR(s)^{-2}}},\\
&\dot{\mcW}(s)= \frac{\mcL(s)}{\mcR(s)^3 \sqrt{1 + \mcW(s)^2 + \mcL(s) \mcR(s)^{-2}}} + \frac{m(s, \mcR(s))}{\mcR(s)^2},\\
&\dot{\mcL}(s)= 0
\end{aligned}
\right.
\end{equation}
both of which are augmented by initial conditions
\begin{equation}
\label{charanginit}
\mcR(\tau) = r, \qquad \mcW(\tau) = w, \qquad \mcL(\tau) = \ell.
\end{equation}
Note that $\mcL(s) = \ell$ for every $s \geq 0$ because the angular momentum of particles is conserved in time along characteristics.

In order to precisely state the main results, we first define notation for the support of $f$ and the maximal particle position and radial momentum on this set. For every $t \geq 0$, define
$$S(t) = \left \{ (r, w, \ell) : f(t, r, w, \ell) > 0 \right \},$$
and
$$\bbS(t) = \overline{S(t)},$$
as well as
$$\mfR(t) =\sup_{(r, w, \ell) \in \bS(0)} \mcR(t, 0, r, w, \ell),$$ 
and
$$\mfW(t) =\sup_{(r, w, \ell) \in \bS(0)} \left | \mcW(t, 0, r, w, \ell) \right |.$$ 
We will also make some use of projection operators, so define
$$\pi_r(\bS(t)) = \{r \geq 0 : (r,w,\ell) \in \bS(t)\}$$
for every $t \geq 0$, with analogous notation for $\pi_w$ and $\pi_\ell$.
In the current paper we prove that the optimal rate of decay is attained for all suitable $L^p$ norms of the electric field and charge density, as well as, the maximal position and radial momentum on the support of $f$, namely
\begin{theorem}
\label{T1}
Let $f_0 \in C_c^1(\mathbb{R}^6)$ be spherically-symmetric and not identically zero. Then, for any $p \in \left (\frac{3}{2},\infty \right ]$ and $q \in \left [1,\infty \right ]$ there are $C_1, C_2 > 0$ such that for $t \geq 0$ the solution of \eqref{VP} or \eqref{RVP} satisfies
$$
\begin{gathered}
C_1 (1+t)^{-2 + \frac{3}{p}} \leq \Vert E(t) \Vert_p \leq C_2 (1+t)^{-2 + \frac{3}{p}}, \\
\\
C_1 (1+t)^{-3 + \frac{3}{q}} \leq \Vert \rho(t) \Vert_q \leq C_2 (1+t)^{-3 + \frac{3}{q}}, \\
\\
C_1 \leq \mfW(t) \leq C_2,\\
\\
C_1(1+t) \leq \mfR(t) \leq C_2(1+t).
\end{gathered}
$$
\end{theorem}

In addition, we show that every momentum characteristic has a limit as $t \to \infty$ and use this to obtain the leading order behavior of characteristics that is equivalent to the asymptotic behavior of the repulsive $N$-body problem shown in \cite{ReinSIMA}.

\begin{theorem}
\label{T2}
Let $f_0 \in C_c^1(\mathbb{R}^6)$ be spherically-symmetric and consider solutions of \eqref{VP} or \eqref{RVP}. Let $\mcU$ be a compact subset of $[0,\infty)\times \bfR \times [0,\infty)$. Then, for any $(r,w,\ell) \in \mcU$ the limiting function $\mcW_\infty$ defined by
$$\mcW_\infty(r, w, \ell) :=  \lim_{t \to \infty} \mcW(t, 0, r, w, \ell)$$ 
exists and is bounded, nonnegative, and continuous.
Additionally, there is $C > 0$ such that for any $(r,w,\ell) \in \mcU$,
$$\left | \mcW(t, 0, r, w, \ell) - \mcW_\infty(r, w, \ell)  \right | \leq C(1+t)^{-1}.$$
For \eqref{VP}, we further find for any $(r,w,\ell) \in \mcU$
$$\mcR(t, 0, r, w, \ell) =  r + \mcW_\infty(r, w, \ell)t + \mathcal{O}\left (\ln \left (1+t \right) \right), $$
while for \eqref{RVP} we have
$$\mcR(t, 0, r, w, \ell) =  r + \frac{\mcW_\infty(r, w, \ell)}{\sqrt{1 + \mcW_\infty(r, w, \ell)^2}}t + \mathcal{O}\left (\ln \left (1+t\right) \right). $$

\end{theorem}

Finally, we utilize these momentum limits to identify the limiting behavior of the spatial average of the particle distribution as $t \to \infty$. To prove this result, we will make an assumption on the set of initial angular momenta, namely that there is $C> 0$ such that 
\begin{equation} 
\label{Condition} 
\tag{A} 
\inf \pi_\ell(\bS(0)) = \inf \{ \ell \geq 0: (r,w,\ell) \in \bS(0)\} \geq C. 
\end{equation}

\begin{theorem}
\label{T3}
Let $f_0 \in C_c^1(\mathbb{R}^6)$ be spherically-symmetric and let $f(t,r,w,\ell)$ be the corresponding solution of \eqref{VP} or \eqref{RVP}. 
Define
$$\Omw = \left \{ \mcW_\infty(0, r, w, \ell) : (r, w, \ell) \in \bS(0) \right \},$$
$$\Oml = \left \{ \ell \in [0,\infty) : (r, w, \ell) \in \bS(0) \right \},$$
and
$\Om = \Omw \times \Oml$.
If $\bS(0)$ satisfies \eqref{Condition}, then there is $F_\infty \in C_c^1(\bfR \times [0,\infty))$ supported on $\overline{\Om}$ such that the spatial average
$$F(t,w, \ell) = \int_0^\infty f(t,r,w,\ell) \ dr$$
satisfies $F(t,w,\ell) \to F_\infty(w,\ell)$ uniformly in $(w,\ell)$ as $t \to \infty$.
%
In particular, we have
$$4\pi^2 \iint\limits_{\Om} F_\infty(\winf,\linf) \ d\linf d\winf = \mcM$$
and for every $\phi \in L^1_{\mathrm{loc}}(0,\infty)$ ,
$$\iint\limits_{\Om} \phi(\linf) F_\infty(\winf,\linf) \ d\linf d\winf = \mcJ_\phi.$$
Additionally, for \eqref{VP} the limiting distribution satisfies
$$2\pi^2 \iint\limits_{\Om} \winf^2 F_\infty(\winf, \linf) \ d\linf d\winf = \mcEVP,$$
while for \eqref{RVP}, it satisfies
$$4\pi^2 \iint\limits_{\Om} \sqrt{1 + \winf^2} F_\infty(\winf, \linf) \ d\linf d\winf = \mcERVP.$$
\end{theorem}
We note that the optimal decay rates of the charge density and electric field in $L^\infty$ are the instrumental components in constructing $\finf$, and the symmetry assumption is merely used to obtain these rates.
Therefore, analogous tools could be used to prove such a theorem without the assumption of spherical symmetry.
However, the fastest rate of decay currently known \cite{Yang} for the electric field in \eqref{VPgeneral} is
$$\| E(t) \|_\infty \leq C(1+t)^{-1/6},$$
and this is not enough to prove the existence of limiting momenta.
Furthermore, a sufficiently fast rate of decay cannot be obtained in lower dimensions \cite{BKR}.

The proofs of Theorems~\ref{T1}-\ref{T3} are obtained by collecting the results of a variety of lemmas within subsequent sections.
In particular, Theorem \ref{T2} is immediately implied by the conclusions of Lemmas \ref{L6} and \ref{L8} (with $\tau = 0$), and the same is true of Theorem \ref{T3} vis-a-vis Lemmas \ref{L7}-\ref{FCconv}.
To focus on the content of the lemmas, we place the proof of Theorem \ref{T1} within the final section. 
For the remainder of the paper we will inherently assume $f_0 \in C_c^1(\mathbb{R}^6)$ is spherically-symmetric. Throughout, the constant $C > 0$ may change from line to line, but when it is necessary to denote a certain constant, we will distinguish this value with a subscript, e.g. $C_0$.

\section{Convexity Estimates on Characteristics}
\label{characteristics}
We first study the behavior of the systems \eqref{charang} and \eqref{charangrel} that define the characteristics of \eqref{VP} and \eqref{RVP}, respectively.
In particular, we show that all spatial characteristics with positive angular momentum grow like $\mathcal{O}(t)$. Additionally, the corresponding momentum characteristics must eventually assume only positive values for suitably large time.

\begin{lemma}
\label{L1}
Let $\mcU$ be a compact subset of $[0,\infty) \times \bfR \times [0,\infty)$ and let $(r, w, \ell) \in \mcU$ be given with $\ell > 0$. Assume $(\mcR(t), \mcW(t), \ell)$ satisfy \eqref{charang} or \eqref{charangrel} for all $t \geq 0$ and \eqref{charanginit} with $\tau = 0$.
Define
$$D(r, w, \ell) = \sqrt{\frac{\ell}{\ell + r^2w^2}}.$$
Then, we have the following: 
\begin{enumerate}

\item There is $C > 0$ such that
$$\mcR(t)^2 \geq C\ell t^2$$
for every $t \geq 0$.

\item There exists $T_1 = T_1(r, w, \ell) \geq 0$ such that
$$\mcW(t) = \dot{\mcR}(t) > 0$$
for all $t \in (T_1, \infty)$.
Furthermore, $T_1 = 0$ for $w \geq 0$ and for solutions of \eqref{charang}
$$ 0 < T_1 \leq \frac{-w r^3}{\ell}$$
for $w < 0$, while
for solutions of \eqref{charangrel}
$$ 0 < T_1 \leq \frac{-w r^3 \sqrt{1 + w^2 + \ell r^{-2}}}{\ell}$$
for $w < 0$.
\item The spatial characteristics attain a minimum at $T_1$ so that
$$\Rm := \min_{t \in [0,\infty)} \mcR(t, 0, r, w, \ell) = \mcR(T_1, 0, r, w, \ell),$$
and this quantity satisfies the bound
$$\Rm \geq
\begin{cases}
r & \mathrm{if} \ w \geq 0 \\
Dr & \mathrm{if} \ w < 0.
\end{cases}$$

\end{enumerate}
\end{lemma}

\begin{proof}
%
Many of the calculations here arise from \cite{BCP1, BCP2, Horst}. Hence, we will provide the full proof for \eqref{charang}, but merely sketch its extension to \eqref{charangrel}.
First, consider solutions of \eqref{charang} and note the convexity of the spatial characteristics.  In particular, we find
\begin{equation}
\label{R2}
\begin{aligned}
\frac{d^2}{dt^2} \left ( \mcR(t)^2 \right ) & = 2 \left (\mcW(t)^2 + \ell \mcR(t)^{-2} \right ) + 2m(t, \mcR(t))\mcR(t)^{-1}\\
& \geq 2 \left (\mcW(t)^2 + \ell \mcR(t)^{-2} \right ).
\end{aligned}
\end{equation}
Similarly, the momentum characteristics satisfy 
\begin{equation}
\label{Winc}
\dot{\mcW}(t) \geq \ell \mcR(t)^{-3} > 0,
\end{equation}
and thus $\mcW(t)$ is increasing.
Finally, the square of the particle speed satisfies
\begin{equation}
\label{V2}
\frac{d}{dt} \left ( \mcW(t)^2 + \ell \mcR(t)^{-2} \right ) = 2\frac{m(t, \mcR(t))}{\mcR(t)^2} \mcW(t).
\end{equation}

We will focus on establishing the lower bound
\begin{equation}
\label{Rlb}
\mcR(t)^2 \geq \ell r^{-2} t^2
\end{equation}
for $t \geq 0$.
Consider $w \geq 0$. Then, by \eqref{Winc} it follows that $\mcW(t) > w \geq 0$ for all $t \geq 0$.
Therefore, \eqref{V2} and the nonnegativity of the enclosed mass implies  
$$\frac{d}{dt} \left ( \mcW(t)^2 + \ell \mcR(t)^{-2} \right ) \geq 0$$
for all $t \geq 0$, and because this function is increasing \eqref{R2} yields
$$\frac{d^2}{dt^2} \left ( \mcR(t)^2 \right ) \geq 2(w^2 + \ell r^{-2}).$$
Integrating in $t$ twice then implies
$$\mcR(t)^2 \geq r^2 + 2rwt + \left (w^2 + \ell r^{-2} \right )t^2 = (r+wt)^2 +  \ell r^{-2} t^2 \geq \ell r^{-2} t^2$$
which provides the stated lower bound \eqref{Rlb}.
We further find $\mcR(t) \geq r$ for all $t \in [0,\infty)$, and thus define the time $T_1$ at which the spatial minimum occurs to be identically zero.

Now, instead consider $w < 0$. Then, define
$$\Tm = \sup \{t \geq 0 : \mcW(t) \leq 0 \}$$
and note that $w < 0$ implies $\Tm > 0$.
We first show that $\Tm < \infty$. For the sake of contradiction, assume $\Tm =\infty$. Then, we have $\dot{\mcR}(t) = \mcW(t) \leq 0$ for all $t \geq 0$ and thus $\mcR(t) \leq r$ for all $t \geq 0$.  From \eqref{charang} and the nonnegativity of the enclosed mass, we find
$$ \ddot{\mcR}(t)  = \frac{\ell}{R(t)^3} + \frac{m(t, \mcR(t))}{\mcR(t)^2}\geq \ell \mcR(t)^{-3} \geq \ell r^{-3},$$
and upon integrating this yields 
$$\mcW(t) \geq \ell r^{-3}t + w$$
for all $t \geq 0$.
Taking $t > \frac{-w r^3}{\ell}$ implies $\mcW(t) > 0$, thus contradicting the assumption that $\Tm = \infty$, and we conclude that $\Tm$ must be finite. Of course, upon deducing $T_0 < \infty$ this same argument holds on the interval $[0,T_0]$ and further implies the upper bound
$$T_0 \leq \frac{-w r^3}{\ell}.$$

Since $\Tm < \infty$ and $\dot{\mcW}(t) > 0$ for $t \geq 0$, we find 
$$
\left \{
\begin{gathered}
\mcW(t)  < 0 \ \mathrm{for} \ t \in [0,\Tm),\\
\mcW(\Tm)  =0, \ \mathrm{and}\\
\mcW(t)  > 0 \ \mathrm{for} \ t \in (\Tm,\infty). 
\end{gathered}
\right.$$
Additionally, \eqref{V2} shows that
$$\frac{d}{dt} \left ( \mcW(t)^2 + \ell \mcR(t)^{-2} \right ) \biggr \vert_{t = \Tm} = 2\frac{m(\Tm, \mcR(\Tm))}{\mcR(\Tm)^2} \mcW(\Tm) = 0$$
and implies that both $\mcR(t)^2$ and $\mcW(t)^2 + \ell \mcR(t)^{-2}$ are minimized at $\Tm$, as the derivative of each of these quantities changes from negative to positive at $t = \Tm$.
Thus, we define
$$\Rm^2 := \min_{t \geq 0} \mcR(t)^2 = \mcR(\Tm)^2$$
and
$$\Vm^2 := \min_{t \geq 0} \left ( \mcW(t)^2 + \ell \mcR(t)^{-2} \right )= \ell \mcR(\Tm)^{-2}.$$
The identity
\begin{equation}
\label{lrep}
\ell = \Rm^2 \Vm^2
\end{equation}
then follows immediately.

From \eqref{R2} we find
$$\frac{d^2}{dt^2} ( \mcR(t)^2 ) \geq 2  \left ( \mcW(t)^2 + \ell \mcR(t)^{-2} \right ) \geq 2\Vm^2$$
for all $t \geq 0$.
Integrating twice in time yields
\begin{eqnarray*}
\mcR(t)^2 & \geq & \mcR(\Tm)^2 + 2\mcR(\Tm)\mcW(\Tm) (t - \Tm) + \Vm^2 (t - \Tm)^2\\
& = & \Rm^2 + \Vm^2 (t - \Tm)^2
\end{eqnarray*}
for any $t \geq 0$. In particular, evaluating this expression at $t = 0$ gives
\begin{equation}
\label{rgeq}
r^2 \geq \Rm^2 + \Vm^2 \Tm^2.
\end{equation}
Returning to the original lower bound for $\mcR(t)^2$, we divide by $t^2$ to find
$$\frac{\mcR(t)^2}{t^2} \geq \frac{\Rm^2 + \Vm^2 (t - \Tm)^2}{t^2}.$$
The right side of this inequality can then be minimized over all $t > 0$, and we find
$$\frac{\Rm^2 + \Vm^2 (t - \Tm)^2}{t^2} \geq \frac{\Rm^2 \Vm^2}{\Rm^2 + \Vm^2 \Tm^2},$$
which yields the subsequent lower bound
\begin{equation}
\label{R2lower}
\mcR(t)^2 \geq \frac{\Rm^2 \Vm^2}{\Rm^2 + \Vm^2 \Tm^2} t^2.
\end{equation}
Now, using \eqref{lrep} and \eqref{rgeq} in \eqref{R2lower}, we find
$$\mcR(t)^2 \geq \frac{\ell}{\Rm^2 + \Vm^2 \Tm^2} t^2 \geq \ell r^{-2}t^2$$
and the desired lower bound \eqref{Rlb} is again achieved, which establishes this inequality for all $w$.
With this, the first result merely follows from the bound on $r$ as using
$r \leq C$
within \eqref{Rlb} yields
$$\mcR(t)^2 \geq C\ell t^2.$$

Turning to the stated bounds on $\Rm$ and $T_1$ for $w < 0$, we note that because $\mcW(t) \leq 0$ for $t \in [0,T_0]$, equation \eqref{V2} implies
$$ \mcW(t)^2 + \ell \mcR(t)^{-2} \leq w^2 + \ell r^{-2} $$
for all $t \in [0,T_0]$. 
Evaluating this expression at $t = T_0$ yields
$$\ell \Rm^{-2} \leq w^2 + \ell r^{-2}$$
and rearranging, we find
$$\Rm \geq r \sqrt{\frac{\ell}{\ell + r^2 w^2}} = D r.$$
The remaining results for solutions of \eqref{charang} then follow by defining $T_1 = 0$ if $w \geq 0$ and $T_1 = \Tm$ if $w < 0$.  
 
The proof for characteristics of \eqref{RVP} uses analogous tools, and we merely sketch it. 
As before, the convexity of the spatial characteristics is crucial, and we find
\begin{equation*}
\frac{1}{2}\frac{d^2}{dt^2} \left ( \mcR(t)^2 \right ) 
\geq\frac{\mcW(t)^2 + \ell \mcR(t)^{-2}}{1 + \mcW(t)^2 + \ell \mcR(t)^{-2}} >0.
\end{equation*}
Similar to \eqref{V2}, the particle rest velocity satisfies
\begin{equation}
\label{V2rel}
\frac{d}{dt} \sqrt{1 + \mcW(t)^2 + \ell\mcR(t)^{-2}} = m(t,\mcR(t)) \mcR(t)^{-2} \frac{\mcW(t)}{\sqrt{1  + \mcW(t)^2 + \ell\mcR(t)^{-2}}}.
\end{equation}
Using this, it is straightforward to establish the relativistic analogue of \eqref{Rlb}, namely
\begin{equation}
\label{Rlbrel}
\mcR(t)^2 \geq \frac{\ell r^{-2}}{1 + w^2 + \ell r^{-2}} t^2
\end{equation}
for $t \geq 0$.
The only new ingredient is the increasing nature of the function $g(x) = \frac{x}{1+x}$, which ensures
$$ g\left (\mcW(t)^2 + \ell\mcR(t)^{-2} \right ) \geq g(w^2 + \ell r^{-2})$$
when $w \geq 0$
and the opposite inequality when $w < 0$. 
With \eqref{Rlbrel} estabilshed, the first estimate follows
as $(r,w,\ell) \in \mcU$ implies
$$\frac{r^{-2}}{1 + w^2 + \ell r^{-2}} = \frac{1}{r^2(1+w^2) + \ell} \geq C.$$
and the remaining bounds follow as before. We omit the details for brevity.
\end{proof}


\section{Field Estimates}

Next, we prove the decay of the field in $L^\infty$. The unifying result here is the lower bound
$$\mcR(t,0, r, w, \ell)^2 \geq C\ell t^2$$
for $(r,w,\ell) \in S(0)$ and $\ell > 0$, which is achieved for either system due to the compact support of $f_0$ and Lemma \ref{L1}.

\begin{lemma}
\label{L2}
For any $f_0 \in C^1_c(\mathbb{R}^6)$, there is $C > 0$ such that
$$\Vert E(t) \Vert_\infty \leq C(1+t)^{-2}$$
for any $t \geq 0$.
%
\end{lemma}

\begin{proof}
We first estimate the enclosed mass. The Vlasov equation implies for every $t \geq 0$ and $(r,w,\ell) \in \bS(t)$
$$ f(t,r, w, \ell) = f_0(\mcR(0, t, r, w, \ell), \mcW(0, t, r, w, \ell), \ell).$$
With this, we find for any $R > 0$
\begin{eqnarray*}
m(t,R) 
& = & 4\pi^2 \iiint\limits_{\bS(t)} f(t,r, w, \ell) \chfn_{\{r \leq R\}}  \ d\ell dw dr\\
& = & 4\pi^2 \iiint\limits_{\bS(t)} f_0(\mcR(0, t, r, w, \ell), \mcW(0, t, r, w, \ell),\ell) \chfn_{ \{r^2 \leq R^2\} }\ d\ell dw dr\\
& = & 4\pi^2 \iiint\limits_{\bS(0)} f_0(\tilde{r}, \tilde{w}, \tilde{\ell}) \chfn_{\{\mcR(t, 0, \tilde{r}, \tilde{w}, \tilde{\ell})^2 \leq R^2 \} }  \ d\tilde{\ell} d\tilde{w} d\tilde{r}
\end{eqnarray*} 
where, in the last equality, we have used the change of variables
$$ \left \{
\begin{gathered}
\tilde{r} = \mcR(0, t, r, w, \ell)\\
\tilde{w} = \mcW(0, t, r, w, \ell)\\
\tilde{\ell} = \mcL(0, t, r, w, \ell) = \ell
\end{gathered}
\right. $$ 
with inverse mapping
$$ \left \{
\begin{gathered}
r = \mcR(t, 0, \tilde{r}, \tilde{w}, \tilde{\ell})\\
w = \mcW(t, 0, \tilde{r}, \tilde{w}, \tilde{\ell})\\
\ell = \mcL(t, 0, \tilde{r}, \tilde{w}, \tilde{\ell}) = \tilde{\ell}
\end{gathered}
\right. $$ 
and the well-known measure-preserving property (cf. \cite{Glassey}) which ensures
$$\left \vert \frac{\partial(r, w, \ell)}{\partial (\tilde{r}, \tilde{w}, \tilde{\ell})} \right \vert = 1.$$
Due to Lemma \ref{L1} we find
$$\mcR(t, 0, \tilde{r}, \tilde{w}, \tilde{\ell})^2 \geq C\tilde{\ell} t^2$$
for $(\tilde{r}, \tilde{w}, \tilde{\ell}) \in S(0)$ with $\tilde{\ell} > 0$, and thus
$$\{ (\tilde{r}, \tilde{w}, \tilde{\ell}) \in \bS(0) : \mcR(t, 0, \tilde{r}, \tilde{w}, \tilde{\ell})^2 \leq R^2, \tilde{\ell} > 0\} \subseteq \{ (\tilde{r}, \tilde{w}, \tilde{\ell}) \in \bS(0) : 0 < \tilde{\ell} \leq CR^2t^{-2} \}.$$
Using this with the $L^\infty$ bound on initial data produces the upper bound
\begin{eqnarray*}
m(t,R)  & \leq & 4\pi^2 \iiint\limits_{\bS(0)} f_0(\tilde{r}, \tilde{w}, \tilde{\ell}) \chfn_{\{0 <\tilde{\ell} \leq CR^2t^{-2} \} }  \ d\tilde{\ell} d\tilde{w} d\tilde{r}\\
& \leq & C \int_0^{\mfR(0)} \int_{-\mfW(0)}^{\mfW(0)} \int_0^{CR^2t^{-2}} f_0(\tilde{r}, \tilde{w}, \tilde{\ell}) \ d\tilde{\ell} d\tilde{w} d\tilde{r}\\
& \leq & CR^2t^{-2}.
\end{eqnarray*} 
With this, we have
$$|E(t,x)| = \frac{m(t,r)}{r^2} \leq Ct^{-2}$$
for every $t > 0, x \in \mathbb{R}^3$, and thus
\begin{equation}
\label{Edecay}
\Vert E(t) \Vert_\infty \leq Ct^{-2}.
\end{equation}
The stated bound then follows as there is $C > 0$ such that $\Vert E(t) \Vert_\infty \leq C$ for $t \in [0,1]$ by the global existence theorem.
\end{proof}

%

Now that we have obtained an upper bound on the supremum of the field, we turn our attention to estimating this quantity from below.
\begin{lemma}
\label{Ebelow}
For any nontrivial $f_0 \in C^1_c(\mathbb{R}^6)$, there is $C > 0$ such that
$$\Vert E(t) \Vert_\infty \geq C\mfR(t)^{-2}$$
for all $t \geq 0$.
\end{lemma}
\begin{proof}

We begin by representing the mass along the largest spatial characteristic on $S(t)$.
Using the Vlasov equation and the change of variables as in the proof of Lemma \ref{L2}, 
it follows that for any $t \geq 0$
\begin{eqnarray*}
\int_0^{\mfR(t)} \int_{-\infty}^\infty \int_0^\infty f(t, r, w, \ell) \ d\ell dw dr
& = &\int_0^{\mfR(t)} \int_{-\infty}^\infty \int_0^\infty f_0 (\mcR(0, t, r,w,\ell), \mcW(0,t,r,w,\ell),  \ell) \ d\ell dw dr\\
&= & \int_0^{\mfR(0)} \int_{-\infty}^\infty \int_0^\infty  f_0(\tilde{r}, \tilde{w}, \tilde{\ell}) \ d\tilde{\ell} d\tilde{w} d\tilde{r}.
\end{eqnarray*}
Inserting this into the representation of the enclosed mass, we have
\begin{eqnarray*}
m(t, \mfR(t)) & = & 4\pi^2 \int_0^{\mfR(t)} \int_{-\infty}^\infty \int_0^\infty  f(t, r, w, \ell) \ d\ell dw dr\\
&  = & 4\pi^2 \int_0^{\mfR(0)} \int_{-\infty}^\infty \int_0^\infty f_0(\tilde{r}, \tilde{w}, \tilde{\ell}) \ d\tilde{\ell} d\tilde{w} d\tilde{r}\\
& = & \mcM.
\end{eqnarray*}

Thus, due to the field representation we find for any $t \geq 0$
$$\vert E(t, \mfR(t))\vert = \frac{m(t, \mfR(t))}{\mfR(t)^2} = \mcM \mfR(t)^{-2}.$$
Because $f_0$ is nontrivial, we conclude $\mcM \neq 0$.
Finally, since $|E(t,x)|$ attains this value at some $x \in \mathbb{R}^3$, we have
$$\Vert E(t) \Vert_\infty \geq C\mfR(t)^{-2}$$
for $t \geq 0$ and the proof is complete.
 \end{proof}

Next, we estimate the field in $L^p$ for $\frac{3}{2} < p < \infty$. In particular, this will yield the optimal decay rate for the potential energy (i.e., $\frac{1}{2}\| E(t) \|_2^2$), an upper bound for which was previously known for any solution of \eqref{VPgeneral} even without spherically-symmetric initial data \cite{GStrauss,IR,Perthame}. However, a similar upper bound for \eqref{RVP} had not previously been obtained.

\begin{lemma}
\label{PotDecay}
For any nontrivial $f_0 \in C^1_c(\mathbb{R}^6)$ and $p \in \left ( \frac{3}{2}, \infty \right)$, there are $C_1, C_2 > 0$ such that
$$C_1 \mfR(t)^{-2p + 3} \leq \int |E(t,x) |^p \ dx \leq C_2(1+t)^{-2p +3}$$
for $t\geq0$.
\end{lemma}
\begin{proof}
The upper bound is obtained by familiar tools.  Indeed, we decompose the field integral as
$$\int |E(t,x) |^p \ dx = \int_{|x | < R} |E(t,x)|^p dx + 4\pi \int_R^\infty \frac{m(t,r)^p}{r^{2p}} r^2dr =: A + B$$
and estimate
$$A \leq 4\pi \|E(t) \|_\infty^p \int_0^R r^2 dr = \frac{4\pi}{3}R^3 \|E(t) \|_\infty^p,$$
while $B$ satisfies
$$B \leq 4\pi \mcM^p \int_R^\infty r^{-2p+2} dr \leq C \mcM^pR^{-2p+3}$$
for $p > \frac{3}{2}$.
Optimizing in $R$ yields $R = C\| E(t) \|_\infty^{-\frac{1}{2}}$ and 
$$\int |E(t,x) |^p \ dx \leq C\| E(t) \|_\infty^{p-\frac{3}{2}}.$$
Finally, due to Lemma \ref{L2} we conclude
$$\int |E(t,x) |^p \ dx \leq C(1+t)^{-2p+3}$$
for $t \geq 0$.

Next, we prove the lower bound. In particular, using the definition of the field and the maximal spatial support of $f$, we find
$$\int | E(t,x) |^p \ dx = 4\pi \int_0^\infty m(t,r)^p r^{2-2p} \ dr \geq 4\pi \int_{\mfR(t)}^\infty m(t,r)^p r^{2-2p} \ dr.$$
For $r \geq \mfR(t)$, we note that
$m(t,r) = \mcM$ as shown in the proof of Lemma \ref{Ebelow}. Thus, we have
$$ \int_{\mfR(t)}^\infty m(t,r)^p r^{2-2p} \ dr = \mcM^p \int_{\mfR(t)}^\infty r^{2-2p} \ dr = \frac{\mcM^p}{2p-3}\mfR(t)^{-2p+3}$$
for $p > \frac{3}{2}$.
As $\mcM \neq 0$, this implies
$$\int | E(t,x) |^p\ dx \geq C\mfR(t)^{-2p+3}$$
for any $t \geq 0$.
\end{proof}


\section{Asymptotics of the Maximal Support Functions}

With strong estimates for $E$, we may now prove the optimal growth rates of the maximal support functions of $f$, and ultimately use them to complete the estimate of the field from below.

\begin{lemma}
\label{L3}
For any nontrivial $f_0 \in C^1_c(\mathbb{R}^6)$, there are $C_1, C_2 > 0$ such that for $t \geq 0$
$$C_1 \leq \mfW(t) \leq C_2,$$
$$C_1(1+t)  \leq \mfR(t) \leq C_2 (1+ t),$$
and
$$C_1 \leq \sup_{(r,w,\ell) \in \bS(0)}\sqrt{\mcW(t, 0, r, w, \ell)^2 + \ell \mcR(t, 0, r, w, \ell)^{-2}} \leq C_2.$$

\end{lemma}

\begin{proof}
First, in order to combine the proof for solutions of both systems we define the function
$$A(t, 0, r, w, \ell) =
\begin{cases}
\sqrt{\mcW(t, 0, r, w, \ell)^2 + \ell \mcR(t, 0, r, w, \ell)^{-2}}, &\mbox{for} \ \eqref{VP}\\
\sqrt{1 + \mcW(t, 0, r, w, \ell)^2 + \ell \mcR(t, 0, r, w, \ell)^{-2}}, &\mbox{for} \ \eqref{RVP}.
\end{cases}
$$
Next, we note that by the global-in-time existence theorem \cite{Horst}, the support of $f(t,x,v)$ must remain bounded on any finite time interval $t \in[0,T]$ for any $T > 0$. In particular, there is $C > 0$ such that
$$|\mcX(1,0,x,v)| \leq C \qquad \mathrm{and} \qquad |\mcV(1,0,x,v)| \leq C$$
for all $(x,v)$ in the support of $f_0$.
Using the formulae for radial characteristics, the bounds
$$\mcR(1,0,r,w,\ell) \leq C,$$
$$\sqrt{ \mcW(1,0,r,w,\ell)^2 + \ell \mcR(1,0,r,w,\ell)^{-2}} \leq C,$$
and thus
$$A(1, 0, r, w, \ell) \leq C$$
follow immediately.
From \eqref{V2} and \eqref{V2rel}, we find for either system
$$\left | A'(t) \right |
=  \left | \frac{m(t, \mcR(t))}{\mcR(t)^2} \frac{\mcW(t)}{A(t)} \right |
\leq \frac{m(t, \mcR(t))}{\mcR(t)^2}.$$
Using this inequality and Lemma \ref{L2}, we have for $t \geq 1$
$$A(t) \leq A(1) + \int_1^t \frac{m(s, \mcR(s))}{\mcR(s)^2}  \ ds
\leq C + C \int_1^\infty (1+ s)^{-2} \ ds 
\leq C,$$
which, upon noting that $A(t)$ remains bounded for $t \in [0,1]$ and taking the supremum over $(r,w,\ell) \in S(0)$, yields the third conclusion.

For characteristics of either system, this bound further implies
$$|\mcW(t, 0, r, w, \ell)| \leq C$$
for $t \geq 1$.
Taking the supremum over all $(r,w,\ell) \in \bS(0)$ yields
$$\mfW(t) \leq C$$ 
for $t \geq 1$, and using the boundedness of $\mfW(t)$ for $t \in[0,1]$ then implies the first result.

Finally, we use  \eqref{charang} for the system \eqref{VP} to find for $t \geq 1$
\begin{eqnarray*}
\mcR(t) & \leq & \mcR(1) + \int_1^t \mcW(s) \ ds \\
& \leq & C + \int_1^t C \ ds\\
& \leq & Ct
\end{eqnarray*}
for any spatial characteristic. 
For \eqref{RVP}, we merely note that
$$ \left | \dot{\mcR}(t) \right |  = \left |\hat{\mcW}(t) \right | \leq 1$$
so that $\mcR(t) \leq t$ follows immediately, and a linear growth bound for $t \geq 1$ is achieved in either case.
Finally, by the global existence result $\mcR(t) \leq C$ for $t \in [0, 1]$, and combining this with the previous estimate for $t \geq 1$ and taking the supremum over $(r,w,\ell) \in \bS(0)$ yields
$$\mfR(t) \leq C(1+t)$$
for $t \geq 0$.
The lower bounds on $\mcW(t)$, and hence $\mcR(t)$, are essentially implied by Lemma \ref{L1} as the particles must travel outward from the origin for suitably large times.
We omit a full proof for brevity.
 \end{proof}


\section{Estimates on the Charge Density}
Next, we focus on establishing the optimal decay rate of the charge density.
\begin{lemma}
\label{Lrho}
For any nontrivial $f_0 \in C^1_c(\mathbb{R}^6)$, there are $C_1, C_2 > 0$ such that
$$C_1 \mfR(t)^{-3} \leq \Vert \rho(t) \Vert_\infty \leq C_2(1+t)^{-3}$$
for any $t \geq 0$.

\end{lemma}

\begin{proof}

Now that we have established \eqref{Edecay}, a \emph{nearly} optimal estimate for the charge density follows as in \cite{Horst}.
We sketch a similar proof for solutions of \eqref{VP} here. 
First, with the stated decay rate of the electric field, we can use Cartesian coordinates to show that the momentum support of $f(t)$ is contained in a ball whose diameter is decreasing in time. Indeed, 
let the points $(x,v_1)$ and $(x,v_2)$ lie in the support of $f(t)$ at some time $t \geq 1$. Integrating the characteristic equations and using Lemma \ref{L2} and the boundedness of the spatial support at $t=1$, we find
\begin{eqnarray*}
C \geq \left |\mcX(1,t, x, v_1) - \mcX(1,t,x,v_2) \right | 
& \geq & \left | v_1 - v_2 \right | (t-1) - 2\int_1^t \int_s^t \Vert E(\tau) \Vert_\infty \ d\tau ds\\
& \geq & \left | v_1 - v_2 \right | (t-1) - C \ln(1+t)
\end{eqnarray*}
for $t \geq 1$.
Rearranging the inequality produces
$$|v_1 - v_2| \leq C(t-1)^{-1}(1 + \ln(1+t)) \leq Ct^{-1}(1+\ln(t))$$
for $t \geq 2$.
This estimate shows that for any $t \geq 2$ and $x \in \bfR^3$ there are $v_0 \in \mathbb{R}^3$ and $C  > 0$ such that
$$\left \{ v : f(t,x,v) \neq 0 \right \} \subseteq \left \{ v \in \mathbb{R}^3 : | v - v_0| \leq Ct^{-1}(1+\ln(t)) \right \},$$
and thus
$$\| \rho(t)\|_\infty  =  \sup_{x \in \mathbb{R}^3}\int  f(t,x,v) \ dv \leq \Vert f_0 \Vert_\infty  \sup_{x \in \mathbb{R}^3} \biggl | \left \{ v : f(t,x,v) \neq 0 \right \}  \biggr | \leq Ct^{-3}(1+\ln(t))^3$$
for $t \geq 2$. Coupling this with the boundedness of $\| \rho(t) \|_\infty$ on bounded time intervals then yields
\begin{equation}
\label{rhodecay}
\| \rho(t)\|_\infty  \leq C(1+t)^{-3}(1 + \ln(1+t))^3
\end{equation}
for $t \geq 0$.
As shown in \cite{Horst}, this estimate holds for both \eqref{VP} and \eqref{RVP}.

With this result, the stated decay rate is obtained for \eqref{VP} merely by applying Theorem 1.1 of \cite{Jacknew}. 
Indeed, as shown there, the decay estimate $\|\rho(t) \|_\infty \leq C(1+t)^{-a}$ for $t \geq 0$ and some $a \in (2,3)$ implies $\| \rho(t)\|_\infty \leq C(1+t)^{-3}$ for $t \geq 0$.
While this has not been shown previously for \eqref{RVP}, similar ideas apply, and we will prove the result for this system.


As the estimate \eqref{rhodecay} holds for solutions of \eqref{RVP}, it only remains to remove the logarithmic factor, which is completed as follows.
As shown in \cite{Horst}, the decay of $\|\rho(t)\|_\infty$ implies the same rate of decay for the derivative of the field. In particular, there is $C > 0$ such that
$$\| \nabla_x E(t) \|_\infty \leq C\|\rho(t)\|_\infty \leq C(1+t)^{-3}(1+\ln(1+t))^3$$
for $t\geq 0$. This estimate and Lemma \ref{L2} further imply there is $\eta > 0$ sufficiently small such that
\begin{equation}
\label{DEdecay}
\| \nabla_x E(t) \|_\infty \leq \eta t^{-5/2}
\end{equation}
and
\begin{equation}
\label{Edecay2}
\|E(t) \|_\infty \leq \eta t^{-3/2}
\end{equation}
for all $t \geq T$ and $T$ sufficiently large.
Additionally, we may take $\eta$ as small as desired by choosing $T$ as large as necessary.

Now, because derivatives of the field decay rapidly, the main idea of the small data theorem of \cite{BD} applies. 
In order to utilize these ideas for solutions of \eqref{RVP}, however, we must estimate derivatives of backwards characteristics, and this will be accomplished using Cartesian coordinates.
Let $g : \bfR^3 \to \bfR^3$ be defined by
$$g(v) = \hat{v} = \frac{v}{\sqrt{1 + |v|^2}}$$
and note that for all $v \in \bfR^3$
$$\max_{i=1,2, 3} |\partial_{v_i} g (v) | \leq 1 \qquad \mathrm{and} \qquad \max_{i,j=1, 2, 3}  |\partial_{v_i v_j} g (v) | \leq 1.$$
Then, letting $h(v) = \nabla g(v)$, we find from the characteristic equations and initial conditions
\begin{equation*}
\left \{
\begin{aligned}
& \frac{\partial \dot{\mcX}}{\partial v}(t)= h\left (\mcV(t) \right ) \frac{\partial \mcV}{\partial v}(t),\\
& \frac{\partial \dot{\mcV}}{\partial v}(t)= \nabla_x E(t, \mcX(t)) \frac{\partial \mcX}{\partial v}(t),\\
& \frac{\partial \mcX}{\partial v}(\tau) = 0 ,\qquad \frac{\partial \mcV}{\partial v}(\tau) = 1
\end{aligned}
\right.
\end{equation*}
where $\mcX(t) = \mcX(t, \tau, x, v)$ and similarly for $\mcV(t)$.
Letting 
$$P(t) = \frac{\partial \mcX}{\partial v}(t) + (\tau - t) h(v) \qquad \mathrm{and} \qquad Q(t) = \frac{\partial \mcV}{\partial v}(t) - \mathbb{I},$$
this becomes
\begin{equation*}
\left \{
\begin{aligned}
& \dot{P}(t) = h\left (\mcV(t) \right ) - h(v) + h\left (\mcV(t) \right ) Q(t),\\
& \dot{Q}(t) = \nabla_x E(t, \mcX(t)) \left ( P(t) - (\tau - t) h(v) \right ),\\
\end{aligned}
\right.
\end{equation*}
with $P(\tau) = \dot{P}(\tau) = Q(\tau) = 0.$

Now, consider $T \leq t \leq \tau$ with $T$ sufficiently large. Using the properties of derivatives of $g$, as well as, \eqref{DEdecay} and \eqref{Edecay2}, we find
\begin{eqnarray*}
\left | \dot{P}(t) \right | & \leq & \max_{i,j=1, 2, 3} \|\partial_{v_i v_j} g\|_\infty \left (\left | \mcV(t) - v \right | + |Q(t)| \right )\\
& \leq & \int_t^\tau \left |E(s, \mcX(s)) \right | \ ds + |Q(t)|\\
& \leq & 2\eta t^{-1/2} + |Q(t)|
\end{eqnarray*}
and
$$\left | \dot{Q}(t) \right | \leq \eta t^{-5/2} \left ( |P(t)| + \tau - t \right ).$$
Hence, for $T \leq t \leq \tau$ we integrate and use the initial conditions to find
\begin{eqnarray*}
\left | P(t) \right | & \leq & \int_t^\tau |\dot{P}(s)| \ ds \\
& \leq & 4\eta \left ( \sqrt{1+\tau} - \sqrt{1+t} \right ) + \int_t^\tau |Q(s)| \ ds\\
& \leq & 2\eta (\tau -t) + \eta \int_t^\tau \int_s^\tau u^{-5/2} \left (|P(u)| + \tau - u \right)  \ du ds.
\end{eqnarray*}
Next, define
$$T_0 = \inf \left \{ t \in [T, \tau] : |P(s)| \leq 5\eta(\tau - s) \ \mathrm{for \ all} \ s \in [t,\tau] \right \}$$
and note that $T_0 < \tau$ due to the initial conditions.
Estimating for $t \in [T_0, \tau]$ and integrating by parts twice, we find
\begin{eqnarray*}
\left | P(t) \right | & \leq & 2\eta (\tau -t) + \eta(1+ 5\eta) \int_t^\tau \int_s^\tau u^{-5/2} (\tau - u)  \ du ds\\
& \leq & \left ( 2\eta + \frac{4}{3}\eta (1+ 5\eta)\right )(\tau - t)\\
& \leq & 4\eta (\tau - t)
\end{eqnarray*}
for $\eta$ sufficiently small. Hence, we find $T_0 = T$ and for all $T \leq t \leq \tau$ we have
$$ |P(t)| \leq 5\eta(\tau - t).$$
This then implies
$$ \left | \frac{\partial \mcX}{\partial v}(t, \tau, x, v) + (\tau - t) h(v) \right| \leq 5\eta(\tau - t)$$
on the same time interval.
Therefore, we find
\begin{eqnarray*}
\left |\det \left ( \frac{\partial \mcX}{\partial v}(t, \tau, x, v) \right ) \right | & = & \left | \det \biggl (P(t) -  (\tau - t) h(v) \biggr ) \right |\\
& = & (\tau - t)^3 \left | \det \left ( \frac{P(t)}{\tau - t} -  h(v) \right) \right |
\end{eqnarray*}
for $T \leq t \leq \tau$.
Letting $\eta \to 0$, we see that 
$$\left | \det \left ( \frac{P(t)}{\tau - t} -  h(v) \right) \right | \to \det (h(v)) = \sqrt{1 + |v|^2}.$$
Thus, due to the continuity of the map $A \mapsto \det(A)$, we can take $\eta$ sufficiently small so that
$$\left | \det \left ( \frac{\partial \mcX}{\partial v}(t, \tau, x, v) \right ) \right | \geq \frac{1}{2} (\tau - t)^3 \sqrt{1 + |v|^2} \geq  \frac{1}{2} (\tau - t)^3$$
for $T \leq t \leq \tau$ and $T$ sufficiently large.

With this estimate, we can finally decompose the charge density using the change of variables
$$v \mapsto \mcX(T, t, x, v) =: y$$
so that for $T$ fixed as above and $t\geq 2T$, we have
\begin{eqnarray*}
\rho(t,x) & = & \int f(t,x,v) \ dv\\
& = & \int f(T, \mcX(T,t,x,v), \mcV(T,t,x,v)) \ dv\\
& \leq &  \left ( \frac{1}{2} (t-T)^3 \right )^{-1} \int\limits_{|\mcX(T)|  \leq C} f(T, \mcX(T,t,x,v), \mcV(T,t,x,v)) \left | \det  \left ( \frac{\partial \mcX}{\partial v}(T,t,x,v) \right ) \right | \ dv\\
& \leq & \left ( \frac{1}{2} (t-T)^3 \right )^{-1} \int\limits_{|y| \leq C} \sup_{u \in \bfR^3}  f(T, y, u) \ dy\\
& \leq & C\|f_0\|_\infty (t-T)^{-3}\\
& \leq & Ct^{-3}.
\end{eqnarray*}
Hence, $\| \rho(t) \|_\infty \leq Ct^{-3}$  for $t$ sufficiently large, and the stated upper bound follows.

To establish the lower bound, we merely use the enclosed mass to find for any $t \geq 0$
$$\mcM = m(t, \mfR(t)) = 4\pi^2 \int_0^{\mfR(t)} r^2 \rho(t,r) \ dr
\leq C \Vert \rho(t) \Vert_\infty \mfR(t)^3.$$
Rearranging this inequality then yields the result.
\end{proof}


\section{Momentum Limits and Asymptotic Behavior of Characteristics}

Because the field decays rapidly in time, we can also establish the limiting behavior of the momentum characteristics. Furthermore, an asymptotic approximation for the behavior of spatial characteristics follows immediately.

\begin{lemma}
\label{L6}
Let $\mcU$ be a compact subset of $[0,\infty) \times \bfR \times [0,\infty)$. For any $\tau \geq 0$ and $(r,w,\ell) \in \mcU$, the limiting momentum $\mcW_\infty$ defined by
$$\mcW_\infty(\tau, r, w, \ell) :=  \lim_{t \to \infty} \mcW(t, \tau, r, w, \ell)$$ 
exists, and is bounded, continuous, nonnegative, and invariant under the characteristic flow, namely $\mcW_\infty$ satisfies
\begin{equation}
\label{Winfinv}
\mcW_\infty(t, \mcR(t, \tau, r, w, \ell), \mcW(t, \tau, r, w, \ell), \ell) = \mcW_\infty(\tau, r, w, \ell)
\end{equation}
for any $t \geq 0$.
Additionally, we have the convergence estimates
\begin{equation}
\label{Winfest}
|\mcW(t, \tau, r, w, \ell) - \mcW_\infty(\tau, r, w, \ell) | \leq C(1+t)^{-1}
\end{equation}
for $t \geq 0$
and
$$\left |\mcW_\infty(\tau, r, w, \ell) - w \right | \leq C(1+\tau)^{-1}$$
for $\tau \geq 0$.
Moreover, using the notation $\mcR(s) = \mcR(s, \tau, r, w, \ell)$ and similarly for $\mcW(s)$ we can express this limiting function explicitly. 
For $\ell > 0$ it is
$$\mcW_\infty(\tau, r, w, \ell) = \sqrt{w^2 + \ell r^{-2}} + \int_\tau^\infty \frac{m(s, \mcR(s))}{\mcR(s)^2} \frac{\mcW(s)}{\sqrt{\mcW(s)^2 + \ell \mcR(s)^{-2}}} ds,$$
for \eqref{VP}, while for \eqref{RVP} it satisfies
\begin{align*}
\sqrt{1 + \mcW_\infty(\tau, r, w, \ell)^2} &= \sqrt{1 + w^2 + \ell r^{-2}}\\
& \quad + \int_\tau^\infty \frac{m(s, \mcR(s))}{\mcR(s)^2} \frac{\mcW(s)}{\sqrt{1 + \mcW(s)^2 + \ell \mcR(s)^{-2}}} ds,
\end{align*}
for every $\tau \geq 0$ and $(r, w, \ell) \in \mcU$.
For $\ell = 0$, the limiting function is
$$\mcW_\infty(\tau, r, w, \ell) = w + \int_\tau^\infty \frac{m(s, \mcR(s))}{\mcR(s)^2} ds$$
for either system.
\end{lemma}
\begin{proof}
To begin, we consider particles with no angular momentum. So, let $\ell = 0$ and notice that due to \eqref{charang} and  \eqref{charangrel} such characteristics satisfy
$$\mcW(t,\tau, r, w, 0) = w + \int_\tau^t \frac{m(s, \mcR(s))}{\mcR(s)^2} ds.$$
Thus, define
$$\mcW_\infty(\tau, r, w, 0) = w+ \int_\tau^\infty \frac{m(s, \mcR(s))}{\mcR(s)^2}  ds$$
for every $\tau \geq 0$.
From Lemmas \ref{L2} and \ref{L3}, we find 
$$|\mcW_\infty(\tau)| \leq |w| + \int_\tau^\infty \Vert E(s) \Vert_\infty \ ds \leq C + C(1+ \tau)^{-1} \leq C.$$
Additionally, the convergence estimate
$$| \mcW(t) -\mcW_\infty| =\int_t^\infty \frac{m(s, \mcR(s))}{\mcR(s)^2}  \ ds \leq C(1+t)^{-1}$$ 
holds due to the field decay of Lemma \ref{L2}.
Of course, $\mcW_\infty$ is continuous due to the uniformity of the limit as $t \to \infty$ and the continuity of $\mcW(t)$.
Finally, the convergence estimate in $\tau$ follows by definition as
$$\left |\mcW_\infty(\tau) - w \right |  = \int_\tau^\infty \frac{m(s, \mcR(s))}{\mcR(s)^2} ds \leq C(1+\tau)^{-1}.$$

Next, consider $\ell > 0$. We prove the stated properties for \eqref{VP}, and comment that the analogous results for \eqref{RVP} merely follow by using the rest speed for a relativistic particle, namely
$$ \mcB(t,\tau,r, w, \ell) = \sqrt{1 + \mcW(t,\tau, r, w, \ell)^2 + \ell \mcR(t,\tau, r, w, \ell)^{-2}}$$
and \eqref{V2rel} in the following argument.
First, we define for every $\tau, t \geq 0$, $(r, w, \ell) \in \mcU$ with $\ell > 0$, the classical particle speed
$$ \mcB(t,\tau,r, w, \ell) = \sqrt{\mcW(t,\tau, r, w, \ell)^2 + \ell \mcR(t,\tau, r, w, \ell)^{-2}}$$
so that using \eqref{V2} we can write this function as
$$\mcB(t) = \sqrt{w^2 + \ell r^{-2}} + \int_\tau^t \frac{m(s, \mcR(s))}{\mcR(s)^2} \frac{\mcW(s)}{\sqrt{\mcW(s)^2 + \ell \mcR(s)^{-2}}}  \ ds.$$
Further define
$$\mcB_\infty(\tau, r, w, \ell) = \sqrt{w^2 + \ell r^{-2}} + \int_\tau^\infty \frac{m(s, \mcR(s,\tau, r, w, \ell))}{\mcR(s,\tau, r, w, \ell)^2}  \frac{\mcW(s)}{\sqrt{\mcW(s)^2 + \ell \mcR(s)^{-2}}}  ds$$
which, due to Lemmas \ref{L2} and \ref{L3}, satisfies
$$|\mcB_\infty(\tau)| \leq \sqrt{w^2 + \ell r^{-2}} + \int_\tau^\infty \Vert E(s) \Vert_\infty \ ds \leq C + C(1+ \tau)^{-1} \leq C$$
for every $\tau \geq 0$.
Similarly, we find
\begin{equation}
\label{Bconv}
\vert \mcB(t) - \mcB_\infty \vert \leq \int_t^\infty \frac{m(s, \mcR(s))}{\mcR(s)^2}  \ ds \leq C(1+t)^{-1}
\end{equation}
for $t \geq 0$ and therefore,
$$\lim_{t \to \infty} \mcB(t,\tau, r, w, \ell) = \mcB_\infty(\tau, r, w, \ell).$$
Furthermore, $\mcB_\infty(\tau, r, w, \ell)$ is continuous due to the uniformity of the limit as $t \to \infty$ and the continuity of $\mcB(t, \tau, r, w, \ell)$.

Using Lemma \ref{L1} we find
$$ \ell \mcR(t, \tau, r, w, \ell)^{-2} \leq Ct^{-2}$$
for $t > \tau \geq 0$ and $(r, w, \ell) \in \mcU$.
This directly implies
$$\lim_{t\to\infty}  \ell \mcR(t, \tau, r, w, \ell)^{-2} = 0$$
and, as Lemma \ref{L1} implies $\mcW(t) > 0$ for $t$ sufficiently large, further allows us to define
$$\mcW_\infty(\tau, r, w, \ell) :=  \lim_{t \to \infty} \mcW(t, \tau, r, w, \ell) = \lim_{t \to \infty} \mcB(t, \tau, r, w, \ell) = \mcB_\infty(\tau, r, w, \ell).$$
Additionally, for $t$ sufficiently large we have the estimate
\begin{equation}
\label{WBconv}
|\mcW(t) - \mcB(t) | = \left | \frac{\mcW(t)^2 - \mcB(t)^2}{\mcW(t) + \mcB(t)} \right | = \frac{\ell \mcR(t)^{-2}}{\mcW(t) + \mcB(t)} \leq \frac{\ell \mcR(t)^{-2}}{\ell^{1/2} \mcR(t)^{-1}} \leq Ct^{-1}
\end{equation}
due to Lemma \ref{L1}.
With this, the convergence estimate \eqref{Winfest} follows from \eqref{Bconv} and \eqref{WBconv} by using
$$|\mcW(t) - \mcW_\infty|  = |\mcW(t) - \mcB_\infty| \leq |\mcW(t) - \mcB(t)| +  |\mcB(t) - \mcB_\infty|  \leq C(1+t)^{-1}.$$

The convergence estimate in $\tau$ follows from the formulas obtained for the limiting momenta. 
Indeed, for any $\tau > 0$ and any triple $(r, w, \ell) \in \mcU$ representing a particle at time $\tau$, there is a compact $\mcU_0 \subset [0,\infty)\times \bfR \times[0,\infty)$ and $(\tilde{r}, \tilde{w}, \tilde{\ell}) \in \mcU_0$ with $\ell = \tilde{\ell}$ such that
$$r = \mcR(\tau, 0, \tilde{r},\tilde{w}, \tilde{\ell}) \qquad
\mathrm{and} \qquad w = \mcW(\tau, 0, \tilde{r},\tilde{w}, \tilde{\ell}).$$
Hence, by Lemma \ref{L1} it follows that
$$r = \mcR(\tau, 0, \tilde{r},\tilde{w}, \tilde{\ell}) \geq C\tilde{\ell}^{1/2}\tau = C\ell^{1/2}\tau$$
and
$$w = \mcW(\tau, 0, \tilde{r},\tilde{w}, \tilde{\ell}) > 0$$
by taking $\tau$ sufficiently large.
With this, we find 
\begin{eqnarray*}
\left |\mcW_\infty(\tau, r, w, \ell) - w \right | & \leq & \sqrt{w^2 + \ell r^{-2}} - w  + \int_\tau^\infty \frac{m(s, \mcR(s))}{\mcR(s)^2} ds\\
& \leq & \frac{\ell r^{-2}}{\sqrt{w^2 + \ell r^{-2}} + w} + C(1 + \tau)^{-1}\\
& \leq & C\ell^{1/2} r^{-1} + C(1 + \tau)^{-1}\\
& \leq & C(1 + \tau)^{-1}
\end{eqnarray*}
for $\tau$ sufficiently large, 
and this estimate is extended to all $\tau \geq 0$ as $\mcW_\infty$ and $w$ are bounded for $\tau \geq 0$.


Finally, we show that the limiting momentum of a particle is invariant under the characteristic flow. In particular, for every $\tau, t \geq 0$ and $(r, w, \ell) \in \mcU$, we use the notation $\mcR(t) = \mcR(t, \tau, r, w, \ell)$ and $\mcW(t) = \mcW(t, \tau, r, w, \ell)$
and note that
$$\mcW(s, t, \mcR(t), \mcW(t), \ell) = \mcW(s, \tau, r, w, \ell) = \mcW(s)$$ 
and similarly for $\mcR(s)$.
Using this identity, we find
\begin{eqnarray*}
\mcW_\infty(t, \mcR(t), \mcW(t),\ell) & = & \sqrt{\mcW(t)^2 + \ell \mcR(t)^{-2}} + \int_t^\infty \frac{m(s, \mcR(s))}{\mcR(s)^2} \frac{\mcW(s)}{\sqrt{\mcW(s)^2 + \ell \mcR(s)^{-2}}} ds\\
& = & \sqrt{\mcW(t)^2 + \ell \mcR(t)^{-2}} + \int_t^\tau \frac{m(s, \mcR(s))}{\mcR(s)^2} \frac{\mcW(s)}{\sqrt{\mcW(s)^2 + \ell \mcR(s)^{-2}}} ds\\
& \ & \ + \int_\tau^\infty \frac{m(s, \mcR(s))}{\mcR(s)^2} \frac{\mcW(s)}{\sqrt{\mcW(s)^2 + \ell \mcR(s)^{-2}}} ds\\
& = & \sqrt{\mcW(\tau)^2 + \ell \mcR(\tau)^{-2}} + \int_\tau^\infty \frac{m(s, \mcR(s))}{\mcR(s)^2} \frac{\mcW(s)}{\sqrt{\mcW(s)^2 + \ell \mcR(s)^{-2}}} ds\\
& = & \sqrt{w^2 + \ell r^{-2}} + \int_\tau^\infty \frac{m(s, \mcR(s))}{\mcR(s)^2} \frac{\mcW(s)}{\sqrt{\mcW(s)^2 + \ell \mcR(s)^{-2}}} ds\\
& = & \mcW_\infty(\tau, r, w, \ell).
\end{eqnarray*}
A similar calculation holds for limiting characteristics with $\ell = 0$, and the proof is complete.
\end{proof}

With the limiting momenta established, we can precisely determine the asymptotic behavior of the spatial characteristics, as well.

\begin{lemma}
\label{L8}
Let $\mcU$ be a compact subset of $[0,\infty) \times \bfR \times [0,\infty)$. For any $\tau \geq 0$ and $(r, w, \ell) \in \mcU$, we find 
$$\lim_{t \to \infty} \frac{\mcR(t,\tau,r, w, \ell) - r}{t - \tau} = \mcW_\infty(\tau, r, w, \ell)$$
for solutions of \eqref{VP}.
In particular, we have for $t \geq \tau \geq 0$ 
$$\mcR(t, \tau, r, w, \ell) =  r + \mcW_\infty(\tau, r, w, \ell) (t-\tau) + \mathcal{O}\left (\ln \left (\frac{1+t}{1 + \tau} \right) \right). $$
For solutions of \eqref{RVP}, we instead have
$$\lim_{t \to \infty} \frac{\mcR(t,\tau,r, w, \ell) - r}{t - \tau} = \widehat{\mcW}_\infty(\tau, r, w, \ell)$$
for any $\tau \geq 0$ and $(r, w, \ell) \in \mcU$, and
$$\mcR(t, \tau, r, w, \ell) =  r + \widehat{\mcW}_\infty(\tau, r, w, \ell) (t-\tau) + \mathcal{O}\left (\ln \left (\frac{1+t}{1 + \tau}\right) \right) $$
for $t \geq \tau \geq 0$,
where
$$ \widehat{\mcW}_\infty(\tau) = \frac{ \mcW_\infty(\tau)}{\sqrt{1 +  \mcW_\infty(\tau)^2}}.$$

\end{lemma}


\begin{proof}

We begin with solutions of \eqref{VP}. Using the convergence estimate in Lemma \ref{L6}, we have
\begin{eqnarray*}
\left | \frac{\mcR(t,\tau, r, w, \ell) - \left [r +  \mcW_\infty(\tau, r, w, \ell) (t-\tau)\right ]}{t- \tau} \right | & \leq & \frac{1}{t - \tau} \int_\tau^t \left | \mcW(s, \tau, r, w, \ell) -  \mcW_\infty(\tau, r, w, \ell) \right | \ ds\\
& \leq & \frac{C}{t - \tau} \int_\tau^t  (1+s)^{-1} ds\\
& = & \frac{C}{t - \tau}\ln\left ( \frac{1+t}{1+\tau} \right)
\end{eqnarray*}
for $0 \leq \tau < t$.
Taking $t \to \infty$ produces the limiting result and multiplying by $t - \tau$ yields the stated asymptotic estimate.
 
Next, we consider solutions of \eqref{RVP}. 
Let $g(x) = \frac{x}{\sqrt{1 + x^2}}$ so that $\widehat{\mcW}_\infty(\tau) = g(\mcW_\infty(\tau))$ and note that $|g(x)| \leq 1$ and 
$$| g'(x) | = (1 + x^2)^{-3/2} \leq 1.$$
Then, we first estimate the difference between the relativistic velocity and its stated limit.
In particular, we have for every $s, \tau \geq 0$
\begin{eqnarray*}
\left | \frac{\mcW(s)}{\sqrt{1 + \mcW(s)^2 + \ell \mcR(s)^{-2}}} - \widehat{\mcW}_\infty(\tau)  \right | & \leq  & \left |  \frac{\mcW(s)}{\sqrt{1 + \mcW(s)^2  + \ell \mcR(s)^{-2}}} - g(\mcW(s))\right | \\
& & + \left | g(\mcW(s)) - g(\mcW_\infty(\tau)) \right |\\
&=: & A + B.
\end{eqnarray*}

Of course, $A = 0$ for $\ell = 0$. For $\ell > 0$, we multiply by the conjugate and use Lemma \ref{L1}, to find 
\begin{eqnarray*}
A 
& = & \left | \frac{g(\mcW(s))}{\sqrt{1 + \mcW(s)^2 + \ell \mcR(s)^{-2}}} \left ( \sqrt{1 + \mcW(s)^2} - \sqrt{1 + \mcW(s)^2 + \ell \mcR(s)^{-2}} \right ) \right |\\ 
& =& \left | \frac{g(\mcW(s))}{\sqrt{1 + \mcW(s)^2 + \ell \mcR(s)^{-2}}} \right | \cdot \left | \frac{\ell \mcR(s)^{-2} }{\sqrt{1 + \mcW(s)^2} + \sqrt{1 + \mcW(s)^2 + \ell \mcR(s)^{-2}}}  \right |\\
& \leq & C\ell \mcR(s)^{-2}\\
& \leq & Cs^{-2}
\end{eqnarray*}
for $s > 0$, and as $A \leq 2 |\mcW(s)| \leq C$ on finite time intervals, we find
$$A \leq C(1+s)^{-2}$$
for all $s \geq 0$.
The estimate for $B$ is straightforward as the convergence result in Lemma \ref{L6} implies
$$B = \left | g\left (\mcW(s) \right)  - g \left (\mcW_\infty(\tau) \right )\right | \leq \| g' \|_\infty  | \mcW(s) - \mcW_\infty(\tau)| \leq C(1+s)^{-1}$$
for $s\geq 0$.
Combining these, we have for $s\geq 0$
\begin{equation}
\label{reldiff}
\left | \frac{\mcW(s)}{\sqrt{1 + \mcW(s)^2 + \ell \mcR(s)^{-2}}} - \widehat{\mcW}_\infty(\tau)  \right | \leq C(1+s)^{-1}.
\end{equation} 

Now, proceeding as in the non-relativistic case and using \eqref{reldiff} with the notation 
$\mcR(t) = \mcR(t, \tau, r, w, \ell)$ and $\mcW_\infty(\tau) = \mcW_\infty(\tau, r, w, \ell)$, we have
\begin{eqnarray*}
\left | \frac{\mcR(t) - \left [r +  \widehat{\mcW}_\infty(\tau) (t-\tau)\right ]}{t- \tau} \right | & \leq & \frac{1}{t - \tau} \int_\tau^t \left | \frac{\mcW(s)}{\sqrt{1 + \mcW(s)^2 + \ell \mcR(s)^{-2}}} - \widehat{\mcW}_\infty(\tau)  \right | \ ds\\
& \leq & \frac{C}{t - \tau} \int_\tau^t  (1+s)^{-1} ds\\
& = & \frac{C}{t - \tau}\ln\left ( \frac{1+t}{1+\tau} \right)
\end{eqnarray*}
for $0 \leq \tau < t$.
As before, taking $t \to \infty$ produces the limiting result and multiplying by $t - \tau$ yields the stated asymptotic estimate.
Notice further that for either system the nonnegativity of $\mcR(t)$ immediately implies that $\mcW_\infty$ must also be nonnegative, as otherwise, the values of $\mcR(t)$ will become negative for sufficiently large $t$.
\end{proof}

\section{Estimates on Derivatives of Characteristics}

In order to construct a limiting particle distribution in subsequent sections, we will first need to estimate derivatives of characteristics.
To establish the remaining results, with the exception of Lemma \ref{L7}, we will now assume \eqref{Condition}, which guarantees that the angular momentum of each particle is bounded away from zero. 
First, we note the improvements in the previous lemmas that result from this assumption.

\begin{lemma}
\label{LAEst}
Assume $\bS(0)$ satisfies \eqref{Condition}. Then, there is $C > 0$ such that for $\tau \geq 0$ with $(r,w,\ell) \in \bS(\tau)$ and $t\geq 0$
$$\mcR(t, \tau, r, w, \ell) \geq Ct.$$
Additionally, there is $C > 0$ such that for $\tau \geq 0$ and $t$ sufficiently large
$$\inf_{(r,w,\ell) \in \bS(\tau)}\mcW(t, \tau, r, w, \ell) \geq C > 0,$$
and this further implies
$$\mcW_\infty(\tau, r, w, \ell) \geq C > 0$$
for $\tau \geq 0$ and $(r,w,\ell) \in \bS(\tau)$.
\end{lemma}

\begin{proof}
The first result follows directly from the lower bounds on $\mcR(t)$ derived in Lemma \ref{L1} and the lower bound on $\ell$ resulting from \eqref{Condition}.
Additionally, \eqref{Condition} implies that the time $T_1$ at which each particle attains its minimal position is uniformly bounded so that
$$\sup_{(r,w,\ell) \in \bS(0)} T_1(r, w, \ell) \leq C =: C_1$$
due to Lemma \ref{L1}.
Because of this, taking $t$ sufficiently large, say $t \geq 2C_1$, implies a uniform lower bound on momenta due to Lemma \ref{L1}.
Finally, the positive lower bound on $\mcW_\infty$ merely follows from the uniform lower bound on $\mcW(t,\tau, r, w, \ell)$ for every $\tau \geq 0$ and $(r,w,\ell) \in \bS(\tau)$.
\end{proof}

Next, we estimate the derivatives of forward characteristics.

\begin{lemma}
\label{LDchar}
Assume $\bS(0)$ satisfies \eqref{Condition}. Then, for $\tau \geq 0$ sufficiently large with $(r,w,\ell) \in \bS(\tau)$ and any $t \geq \tau$, we have
$$\left | \frac{\partial \mcR}{\partial w}(t,\tau, r, w, \ell) \right | \leq C(t-\tau) \qquad \ \mathrm{and} \ \qquad
\left | \frac{\partial \mcW}{\partial w}(t,\tau, r, w, \ell) \right | \leq C.$$
\end{lemma}

\begin{proof}
We first prove the result for solutions of \eqref{VP}. Taking a derivative in \eqref{charang} and \eqref{charanginit} yields
\begin{equation*}
\left \{
\begin{aligned}
& \frac{\partial \ddot{\mcR}}{\partial w}(t)= \left ( \rho(t, \mcR(t)) - 2\frac{m(t, \mcR(t))}{\mcR(t)^3} -3\ell R(t)^{-4} \right ) \frac{\partial \mcR}{\partial w}(t),\\
& \frac{\partial \mcR}{\partial w}(\tau) = 0 ,\qquad \frac{\partial \dot{\mcR}}{\partial w}(\tau) = 1.
\end{aligned}
\right.
\end{equation*}
We denote the term in the parentheses by $P(t, \mcR(t))$. Due to Lemmas \ref{Lrho} and \ref{LAEst} we have
$$\left | P(t, \mcR(t)) \right | \leq C(1+t)^{-3}$$
for $t \geq 0$.
Upon integrating, we can rewrite the solution of this differential equation as
$$\frac{\partial \mcR}{\partial w}(t) = t - \tau + \int_\tau^t (t-s) P(s, \mcR(s)) \frac{\partial \mcR}{\partial w}(s) ds$$
so that
$$\left |\frac{\partial \mcR}{\partial w}(t) \right | \leq  t - \tau  + C\int_\tau^t \frac{t-s}{(1+s)^3} \left | \frac{\partial \mcR}{\partial w}(s) \right | ds$$
follows from the bound on $P$.
Now, we fix some $\delta \in [4,6]$ and define
$$T_0 = \sup \left \{ \bar{t} \geq \tau : \left | \frac{\partial \mcR}{\partial w}(s) \right | \leq \delta (s - \tau) \ \mathrm{for \ all} \ s \in [\tau, \bar{t}] \right \}.$$
Note that $T_0 > \tau$ due to the initial conditions. 
Then, estimating for $t \in [\tau, T_0)$, we have
$$\left |\frac{\partial \mcR}{\partial w}(t) \right | \leq  t - \tau  + C\delta \int_\tau^t \frac{(t-s)(s-\tau)}{(1+s)^3} ds.$$
Integrating by parts twice, we find
$$\left |\frac{\partial \mcR}{\partial w}(t) \right | \leq  t - \tau + C\delta (1+\tau)^{-1} (t-\tau) = (1 + C\delta (1+\tau)^{-1})(t-\tau) \leq 2 (t-\tau) \leq \frac{1}{2}\delta (t-\tau)$$
for $\tau$ sufficiently large. Hence, we find $T_0 =\infty$ and the first estimate follows.
The second estimate is directly implied by the first as
$$\frac{\partial \dot{\mcW}}{\partial w}(t)=P(t, \mcR(t)) \frac{\partial \mcR}{\partial w}(t)$$
so that for $\tau$ sufficiently large
$$\left | \frac{\partial \mcW}{\partial w}(t) \right | \leq 1 + C\int_\tau^t (1+s)^{-3} (s-\tau) \ ds \leq C.$$

Turning to \eqref{RVP}, we first compute the derivatives of spatial and momentum characteristics separately, yielding
$$\frac{\partial \dot{\mcR}}{\partial w}(t)= (1 + \mcW(t)^2 + \ell \mcR(t)^{-2})^{-3/2} \left ( (1 + \ell\mcR(t)^{-2}) \frac{\partial \mcW}{\partial w}(t) + \ell \mcW(t) \mcR(t)^{-3}\frac{\partial \mcR}{\partial w}(t) \right )$$
with $\frac{\partial \mcR}{\partial w}(\tau) = 0$
and
$$\frac{\partial \dot{\mcW}}{\partial w}(t)=P(t, \mcR(t)) \frac{\partial \mcR}{\partial w}(t)$$
with $\frac{\partial \mcW}{\partial w}(\tau) = 1$,
where $P$ is given as above.
Note that the coefficient of $\frac{\partial \mcW}{\partial w}(t)$ in the first of these equations is bounded above by $1$. The inequality
$$\left | \frac{\partial \mcW}{\partial w}(t) \right | \leq 1 + C\int_\tau^t (1+s)^{-3} \left | \frac{\partial \mcR}{\partial w}(s) \right | \ ds$$
follows as before.  Inserting this into the remaining equation and using Lemmas \ref{Lrho} and \ref{LAEst} yields
\begin{eqnarray*}
\left |\frac{\partial \dot{\mcR}}{\partial w}(t) \right | & \leq & \left |\frac{\partial \mcW}{\partial w}(t) \right | + C\mcR(t)^{-3} \left |\frac{\partial \mcR}{\partial w}(t) \right |   \\
& \leq & 1 + C\left (\int_\tau^t (1+s)^{-3} \left |\frac{\partial \mcR}{\partial w}(s) \right | \ ds \right ) + C(1+t)^{-3} \left |\frac{\partial \mcR}{\partial w} (t) \right |.
\end{eqnarray*}
Integrating then gives
$$\left |\frac{\partial \mcR}{\partial w}(t) \right | \leq t-\tau +  C\int_\tau^t  \left (\int_\tau^s (1+u)^{-3} \left |\frac{\partial \mcR}{\partial w}(u) \right | \ du \right ) ds + C\int_\tau^t (1+s)^{-3} \left |\frac{\partial \mcR}{\partial w} (s) \right | \ ds.$$
Similar to the previous proof, we fix some $\delta \in [4,6]$ and define
$$T_0 = \sup \left \{ \bar{t} \geq \tau : \left | \frac{\partial \mcR}{\partial w}(s) \right | \leq \delta (s - \tau) \ \mathrm{for \ all} \ s \in [\tau, \bar{t}] \right \}.$$
Estimating on the interval $[\tau, T_0)$, we have
$$\left |\frac{\partial \mcR}{\partial w}(t) \right | \leq t-\tau +  C\delta \int_\tau^t  \left (\int_\tau^s (1+u)^{-3} (u-\tau) \ du \right ) ds + C\delta \int_\tau^t (1+s)^{-3} (s-\tau)\ ds.$$
Integrating by parts, we find for $s \geq \tau$
$$\int_\tau^s (1+u)^{-3} (u-\tau) \ du \leq C(1+\tau)^{-1}$$
and because $t \geq \tau$
$$\int_\tau^t (1+s)^{-3} (s-\tau)\ ds \leq \frac{1}{2} \left ((1+\tau)^{-1} - (1+t)^{-1} \right ) \leq (1+\tau)^{-2}(t- \tau).$$
Hence, for $\tau$ sufficiently large we have
$$\left |\frac{\partial \mcR}{\partial w}(t) \right | \leq (t-\tau)\left ( 1 +  C\delta(1+\tau)^{-1} + C\delta (1+\tau)^{-2} \right ) \leq 2 (t -\tau) \leq \frac{1}{2}\delta (t -\tau).$$
Therefore, we find $T_0 =\infty$ and the first estimate follows.
The second estimate is directly implied by the first using the same argument as for \eqref{VP}.
\end{proof}

Now that we can control the growth of derivatives of characteristics, we show that the limiting momenta are increasing functions of $w$ for sufficiently large $\tau$, and this property will be useful later to perform a change of variables.

\begin{lemma}
\label{Winfinvert}
Assume $\bS(0)$ satisfies \eqref{Condition}. Then, there is $T_2 > 0$ such that for all $\tau \geq T_2$ and $(r,w,\ell) \in \bS(\tau)$, we have
$$\frac{\partial \mcW_\infty}{\partial w}(\tau, r, w, \ell) \geq \frac{1}{2}.$$
Consequently, for $\tau$ sufficiently large and $(r,w,\ell) \in \bS(\tau)$, the $C^2$ mapping $w \mapsto \mcW_\infty(\tau, r, w, \ell)$ is increasing, injective, and invertible.
\end{lemma}

\begin{proof}
We first consider solutions of \eqref{VP}. Note that the limiting momentum given by Lemma \ref{L6} is continuously differentiable as $w^2 + \ell r^{-2} \geq C > 0$ for $(r,w,\ell) \in \bS(\tau)$ and $\tau$ sufficiently large by Lemma \ref{LAEst}. In particular, we find
$$ \frac{\partial \mcW_\infty}{\partial w}(\tau, r, w, \ell) = \frac{w}{\sqrt{w^2 + \ell r^{-2}}} + \int_\tau^\infty P(s, \tau, r, w, \ell) \ ds $$
where $P$ satisfies
\begin{eqnarray*} 
| P(s) | & = & \biggl | \left [\rho(s,\mcR(s)) - 2 m(s,\mcR(s))\mcR(s)^{-3} \right ] \frac{\mcW(s)}{\sqrt{\mcW(s)^2 + \ell \mcR(s)^{-2}}} \\
& \ & \  + m(s,\mcR(s))\mcR(s)^{-2} \biggl ( (\mcW(s)^2 + \ell \mcR(s)^{-2})^{-3/2} \ell \mcR(s)^{-2}\frac{\partial \mcW}{\partial w}(s)\\
& \ & \  + 2\ell \mcR(s)^{-3}\mcW(s)\frac{\partial \mcR}{\partial w}(s) \biggr ) \biggr |\\
& \leq & C(1+s)^{-3} + C(1+s)^{-2}\left [|\mcW(s)|^{-1} + (1+s)^{-3} |s-\tau| \right ]\\
& \leq & C(1+s)^{-2}
\end{eqnarray*}
for $s \geq \tau$ and $\tau$ sufficiently large due to Lemmas \ref{L3}, \ref{Lrho}, \ref{LAEst}, and \ref{LDchar}.
Additionally, for $(r,w,\ell) \in \bS(\tau)$ and $\tau > 0$, there are $(\tilde{r}, \tilde{w}, \tilde{\ell}) \in \bS(0)$ such that
$$r = \mcR(\tau, 0, \tilde{r},\tilde{w}, \tilde{\ell}) \geq C\tau^{-1}$$
due to Lemma \ref{LAEst}.
Therefore, using the lower bound on momenta from Lemma \ref{LAEst}, we find for $\tau$ sufficiently large and $(r,w,\ell) \in \bS(\tau)$
\begin{eqnarray*}
\left | \frac{w}{\sqrt{w^2 + \ell r^{-2}}}- 1 \right | 
& = & \frac{\ell r^{-2}}{ \left ( w + \sqrt{w^2 + \ell r^{-2}} \right ) \sqrt{w^2 + \ell r^{-2}}}\\
& \leq & \frac{\ell^{1/2} r^{-1}}{w} \leq C\tau^{-1}.
\end{eqnarray*}
With this, we find
\begin{eqnarray*}
\left | \frac{\partial \mcW_\infty}{\partial w}(\tau) - 1 \right | & \leq & \left | \frac{w}{\sqrt{w^2 + \ell r^{-2}}}- 1 \right | + \int_\tau^\infty | P(s) | \ ds\\
& \leq & C(1 + \tau)^{-1}
\end{eqnarray*}
for $\tau$ sufficiently large.
Thus, there is $T_2 > 0$ such that for all $\tau \geq T_2$, we have
$$ \frac{\partial \mcW_\infty}{\partial w}(\tau, r, w, \ell) \geq \frac{1}{2}$$
for $\tau \geq T_2$ and $(r,w,\ell) \in \bS(\tau)$, and the result follows.

To obtain the analogous result for \eqref{RVP}, additional estimates are necessary. We again note that the expression for $\mcW_\infty$ given by Lemma \ref{L6} is $C^1$ and take a derivative to find
\begin{equation}
\label{WinfRVP}
\frac{\mcW_\infty(\tau)}{\sqrt{1 + \mcW_\infty(\tau)^2}} \frac{\partial \mcW_\infty}{\partial w}(\tau) = \frac{w}{\sqrt{1 + w^2 + \ell r^{-2}}} + \int_\tau^\infty P(s) \ ds
\end{equation}
where
$$| P(s) | \leq C(1+s)^{-2}$$
for $s$ sufficiently large by the same argument as for \eqref{VP}.
Now, due to Lemmas \ref{L3} and \ref{L6} we find $$\left | \mcW_\infty(\tau) + w \right | \leq C$$
for any $\tau \geq 0$, and using this with Lemmas \ref{L6} and \ref{LAEst} we find
\begin{eqnarray*}
\left | \frac{w}{\sqrt{1 + w^2 + \ell r^{-2}}} - \frac{\mcW_\infty(\tau)}{\sqrt{1 + \mcW_\infty(\tau)^2}} \right | 
& \leq & |w| \left | \frac{1}{\sqrt{1 + w^2 + \ell r^{-2}}} - \frac{1}{\sqrt{1 + \mcW_\infty(\tau)^2}} \right |\\
& \ & + \frac{\left |\mcW_\infty(\tau) - w \right |}{\sqrt{1 + \mcW_\infty(\tau)^2}}\\
& \leq & C\frac{\left | \sqrt{1 + \mcW_\infty(\tau)^2} - \sqrt{1 + w^2 + \ell r^{-2}}\right |}{\sqrt{1 + w^2 + \ell r^{-2}}\sqrt{1 + \mcW_\infty(\tau)^2}}\\
& \ &  + \left |\mcW_\infty(\tau) - w \right |\\
& \leq & C\left | \mcW_\infty(\tau)^2 - w^2 \right | + C\ell r^{-2} + \left |\mcW_\infty(\tau) - w \right |\\
& \leq & C\left ( | \mcW_\infty(\tau) - w | + \tau^{-2} \right )\\
& \leq & C(1+\tau)^{-1}
\end{eqnarray*}
for $\tau \geq 1$.
Due to Lemma \ref{LAEst}, we conclude
$$\frac{\mcW_\infty(\tau)}{\sqrt{1 + \mcW_\infty(\tau)^2}} \geq C$$
for any $\tau \geq 0$ because the function $x \mapsto \frac{x}{\sqrt{1 + x^2}}$ is increasing.
Finally, using these results with \eqref{WinfRVP} we have
\begin{eqnarray*}
\left | \frac{\partial \mcW_\infty}{\partial w}(\tau) - 1 \right | & = & \left | \frac{\mcW_\infty(\tau)}{\sqrt{1 + \mcW_\infty(\tau)^2}} \frac{\partial \mcW_\infty}{\partial w}(\tau) -  \frac{\mcW_\infty(\tau)}{\sqrt{1 + \mcW_\infty(\tau)^2}}  \right | \cdot \left |  \frac{\mcW_\infty(\tau)}{\sqrt{1 + \mcW_\infty(\tau)^2}}  \right |^{-1}\\
& \leq & C\left | \frac{w}{\sqrt{1 + w^2 + \ell r^{-2}}} - \frac{\mcW_\infty(\tau)}{\sqrt{1 + \mcW_\infty(\tau)^2}} \right | +C \int_\tau^\infty | P(s) | \ ds\\
& \leq & C(1+\tau)^{-1}
\end{eqnarray*}
for $\tau$ sufficiently large.
Hence, as for \eqref{VP} there is $T_2 > 0$ such that for all $\tau \geq T_2$, we have
$$\frac{\partial \mcW_\infty}{\partial w}(\tau) \geq \frac{1}{2}$$
for $\tau \geq T_2$ and $(r,w,\ell) \in \bS(\tau)$.
Finally, we remark that for each system $\mcW_\infty(\tau, r, w, \ell)$ is actually $C^2$, as the term $P(s)$ can be shown to be continuously differentiable and bounded for classical solutions. For brevity, we omit the details.
\end{proof}

\section{Estimates on Derivatives of the Particle Distribution}
With the optimal rate obtained for derivatives of the electric field, we can establish the optimal growth rate for derivatives of the particle distribution.

\begin{lemma}
\label{Df}
There is $C > 0$ such that for any $t \geq 0$,
$$\|\partial_r f(t) \|_\infty \leq C,$$
and
$$\|\partial_w f(t) \|_\infty \leq C(1+t).$$
\end{lemma}

\begin{proof}
We prove the estimate directly for \eqref{VP} and note that the only difference for \eqref{RVP} is the appearance of a $w$-derivative of the relativistic velocity in the resulting expression for $\partial_w f$, and this new term satisfies
$$\left |\partial_w \left ( \frac{w}{\sqrt{1 + w^2 + \ell r^{-2}}} \right ) \right | \leq 1.$$
Taking derivatives in the Vlasov equation and integrating along characteristics, we find for $t \geq 1$
\begin{align*}
& \partial_{w} f(t,r,w,\ell) = \partial_{w} f(1, \mcR(1), \mcW(1), \ell) - \int_1^t \partial_{r} f(s, \mcR(s), \mcW(s), \ell) \ ds\\
& \partial_{r} f(t,r,w,\ell) = \partial_{r} f(1, \mcR(1), \mcW(1), \ell)\\
& \qquad - \int_1^t \left (\rho(s, \mcR(s)) - 2\frac{m(s,\mcR(s))}{\mcR(s)^3} - 3\ell \mcR(s)^{-4} \right ) \partial_w f(s, \mcR(s), \mcW(s), \ell) \ ds.
\end{align*}
The first equation implies
$$ \| \partial_{w} f(t) \|_\infty \leq  \| \partial_{w} f(1) \|_\infty + (t-1) \sup_{s \in [1,t]} \| \partial_{r} f(s) \|_\infty \leq t\mcD(t)$$
where
$$\mcD(t) = 1 + \|\partial_w f(1) \|_\infty + \sup_{s \in [1,t]} \| \partial_r f(s) \|_\infty.$$
Using this estimate with Lemmas \ref{Lrho} and \ref{LAEst}, the second equation gives
\begin{eqnarray*}
\| \partial_r f(t) \|_\infty & \leq &  \| \partial_r f(1) \|_\infty + C\int_1^t (1+s)^{-3}  \| \partial_w f(s) \|_\infty \ ds\\
& \leq & C + C\int_1^t (1+s)^{-2} \mcD(s) \ ds.
\end{eqnarray*}
Taking the supremum in time yields
$$\mcD(t) \leq C + C\int_1^t (1+s)^{-2} \mcD(s) \ ds,$$
and upon invoking Gronwall's inequality, we find
$$\mcD(t) \leq C\exp \left ( \int_1^t (1+s)^{-2} \ ds \right ) \leq C$$
for $t \geq 1$. 
As $\mcD(t)$ is bounded, we find
$$\|\partial_r f(t) \|_\infty \leq C$$
and
$$\|\partial_w f(t) \|_\infty \leq t\mcD(t) \leq C(1+t)$$
for all $t \geq 1$, and these estimates hold for $t \geq 0$.
\end{proof}

\section{Limiting Behavior of the Spatial Average}

With the behavior of the characteristics well-understood as $t \to \infty$, we would also like to understand the distribution of particles in this same limit. Because $f$ is constant along characteristics and Lemma \ref{L8} implies that particle positions diverge almost everywhere on the support of $f$ as $t \to \infty$, it is impossible to represent a pointwise limiting value of $f$ in terms of these variables, and indeed it follows that $f \rightharpoonup 0$ as $t \to \infty$. Hence, it appears that some information may be lost in the limit. However, we know that the radial momentum characteristics converge as $t \to \infty$, while the angular momentum is constant along characteristics, and hence also possesses a limit. Thus, we can construct a particle distribution as a function of these limiting momenta that enables one to recover some microscopic information as $t \to \infty$, and this will represent the limiting behavior of the spatial average.

To do so, we first define the collection of all limiting radial momenta on $\bS(t)$ and show that it is a time-independent set of nonnegative values with compact closure.

\begin{lemma}
\label{L7}
For any $\tau \geq 0$, define the set
$$\Omw(\tau) = \left \{ \mcW_\infty(\tau, r, w, \ell) : (r, w, \ell) \in \bS(\tau) \right \}.$$
Then, for any $s, \tau \geq 0$
$$\Omw(\tau) = \Omw(s).$$
Hence, we define
$$\Omw = \left \{ \mcW_\infty(0, r, w, \ell) : (r, w, \ell) \in \bS(0) \right \}$$
so that $\Omw(\tau) = \Omw$ for all $\tau \geq 0$.
Additionally, $\overline{\Om}_w \subset [0,\infty)$ is a compact interval.
\end{lemma}

\begin{proof}
If $\tau \geq 0$ and $w_\infty \in \Omw(\tau)$, then $w_\infty = \mcW_\infty(\tau, r, w, \ell)$ for some $(r, w, \ell) \in \bS(\tau)$, and thus by \eqref{Winfinv} we have
$$w_\infty = \mcW_\infty(t, \mcR(t, \tau, r, w, \ell),\mcW(t, \tau, r, w, \ell), \ell)$$
for any $t \geq 0$.
Of course, $(r, w, \ell) \in \bS(\tau)$ implies $\left ( \mcR(t, \tau, r, w, \ell), \mcW(t, \tau, r, w, \ell), \ell \right ) \in \bS(t)$
so $w_\infty \in \Omw(t)$, and we find $\Omw(\tau) \subseteq \Omw(t)$. Since this is true for any $\tau, t \geq 0$, we merely swap these values to obtain the reverse inclusion, and it follows that
$$\Omw(\tau) = \Omw(t) = \Omw(0) = \left \{ \mcW_\infty(0, r, w, \ell) : (r, w, \ell) \in \bS(0) \right \}$$
for all $\tau, t \geq 0$.

Finally, in view of Lemma \ref{L1}, $\mcW_\infty (0,r,w,\ell) \geq 0$ for every $(r,w,\ell) \in \bS(0)$ so that $\Omw \subset [0, \infty).$ Additionally, by Lemma \ref{L6} we know $\mcW_\infty (0,r,w,\ell)$ is continuous, and thus its range $\overline{\Om}_w$ on the compact set $\bbS(0)$ must also be compact, which further implies that $\overline{\Om}_w$ is a compact interval.
\end{proof}

Similar to the previous lemma, we could define the limiting angular momentum function $\mcL_\infty(\tau, r, w, \ell)$ for every $\tau \geq 0$ and $(r,w,\ell) \in \bS(\tau)$, but because this quantity is conserved along characteristics, we merely have $\mcL_\infty(\tau, r, w, \ell) = \ell$
and thus, there is no need to use the $\mcL_\infty$ notation.
That being said, the set of all angular momentum limits will still be useful. Hence, analogous to Lemma \ref{L7} we define
$$\Oml(\tau) = \{\mcL_\infty(\tau, r, w, \ell) : (r,w,\ell) \in \bS(\tau)\}$$
and note that this set is a time-independent projection
$$\Oml(\tau) = \Oml = \pi_\ell(\bS(0)) = \{\ell \in [0,\infty) : (r,w,\ell) \in \bS(0)\}$$
for every $\tau \geq 0$ because of the invariance of the angular momentum under the characteristic flow.
Additionally, because $\pi_\ell$ is continuous and $\bbS(0)$ is compact, it follows that $\overline{\Om}_\ell \subset [0,\infty)$ is a compact interval.
With this, we define the product $\Om = \Omw \times \Oml$ and note that $\overline{\Om}$ is compact.
%
Finally, we prove the limiting behavior of the spatial average, which tends to a function of the particle momentum limits.

\begin{lemma}
\label{L9}
Assume $\bS(0)$ satisfies \eqref{Condition}. Then, there exists a particle distribution $F_\infty \in C_c^1(\bfR \times [0,\infty))$ with $\mathrm{supp}(F_\infty) = \overline{\Om}$ such that
$$F(t,w, \ell) = \int_0^\infty f(t,r,w,\ell) \ dr$$
satisfies $F(t,w,\ell) \rightharpoonup F_\infty(w,\ell)$ as a measure as $t \to \infty$, namely
$$\lim_{t \to \infty} \int_{-\infty}^\infty\int_0^\infty \psi(w,\ell) F(t,w,\ell) d\ell dw = \int_{-\infty}^\infty\int_0^\infty \psi(w,\ell) F_\infty(w,\ell) d\ell dw$$
for every $\psi \in C_b \left (\mathbb{R} \times [0,\infty)\right )$.
%
In particular, the limiting distribution satisfies
\begin{enumerate}
\item $\displaystyle 4\pi^2\iint\limits_\Om \finf(\winf, \linf) \ d\linf d\winf = \mcM$
\item For every $\phi \in L^1_{\mathrm{loc}}(0,\infty)$,
$$\iint\limits_\Om \phi(\linf) \finf(\winf, \linf) \ d\linf d\winf = \mcJ_\phi$$ 
\item For solutions of \eqref{VP},
$$2\pi^2 \iint\limits_\Om \winf^2 \finf(\winf, \linf) \ d\linf d\winf = \mcEVP,$$
while for solutions of \eqref{RVP},
$$4\pi^2 \iint\limits_\Om \sqrt{1 + \winf^2} \finf(\winf, \linf) \ d\linf d\winf = \mcERVP.$$
\end{enumerate}
\end{lemma}
\begin{proof}
We first prove the weak limit. 
Let $\psi  \in C_b \left (\mathbb{R} \times [0,\infty)\right )$ be given and fix any $T \geq T_2$ from Lemma \ref{Winfinvert}.
Then, we apply the measure-preserving change of variables, but at time $T$, so that
\begin{eqnarray*}
& \ & \lim_{t \to \infty} \int_{-\infty}^\infty\int_0^\infty \psi(w,\ell) F(t,w,\ell) d\ell dw\\
& = & \lim_{t \to \infty} \int_0^\infty \int_{-\infty}^\infty\int_0^\infty \psi(w,\ell) f(t,r,w,\ell) d\ell dw dr\\
& = & \lim_{t \to \infty} \iiint\limits_{\bS(t)} \psi(w,\ell) f(T, \mcR(T,t,r, w, \ell), \mcW(T,t,r, w, \ell), \ell) \ d\ell dw dr\\
& = & \lim_{t \to \infty}\iiint\limits_{\bS(T)} \psi(\mcW(t,T,\tilde{r}, \tilde{w}, \tilde{\ell}), \tilde{\ell}) f(T, \tilde{r}, \tilde{w}, \tilde{\ell}) \ d\tilde{\ell} d\tilde{w} d\tilde{r}\\
& = & \iiint\limits_{\bS(T)} \psi(\mcW_\infty(T,\tilde{r}, \tilde{w}, \tilde{\ell}), \tilde{\ell}) f(T, \tilde{r}, \tilde{w}, \tilde{\ell}) \ d\tilde{\ell} d\tilde{w} d\tilde{r}.
\end{eqnarray*}
Now, by Lemma \ref{Winfinvert}, for any $(\tilde{r}, \tilde{w}, \tilde{\ell}) \in \bS(T)$, the mapping $\tilde{w} \mapsto \mcW_\infty(T, \tilde{r}, \tilde{w}, \tilde{\ell})$ is $C^2$ with $$\frac{\partial \mcW_\infty}{\partial w}(T, \tilde{r}, \tilde{w}, \tilde{\ell}) \geq \frac{1}{2},$$ and thus bijective from $\pi_w(\bS(T))$ to $\Omega_w$. 
Hence, letting $u = \mcW_\infty(T, \tilde{r}, \tilde{w}, \tilde{\ell})$, we perform a change of variables and drop the tilde notation to find
\begin{align*}
& \lim_{t \to \infty} \int_{-\infty}^\infty\int_0^\infty \psi(w,\ell) F(t,w,\ell) d\ell dw\\
& = \int_0^\infty \int_{-\infty}^\infty \int_0^\infty \psi(u, \ell) f(T, r, w(u), \ell) \frac{\chfn_{\Omega}(u, \ell)}{\frac{\partial \mcW_\infty}{\partial w}(T, r, w(u), \ell)} \ d\ell du dr
\end{align*}
where for fixed $r, \ell$ in the support of $f(T)$, the $C^1$ function $w: \Om_w \to \pi_w(\bS(T)) $ is given by
$$w(u) = \mcW^{-1}_\infty(T,r,u, \ell).$$
Therefore, for any $(u,\ell) \in \bfR \times [0,\infty)$ define 
$$F_\infty(u,\ell) = \int_0^\infty f(T, r, w(u), \ell) \chfn_{\Omega}(u, \ell) \left (\frac{\partial \mcW_\infty}{\partial w}(T, r, w(u), \ell) \right)^{-1} \ dr,$$
and notice
$$| F_\infty(u,\ell) | \leq 2\|f_0\|_\infty \left |\pi_r(\bS(T))\right | \leq C(1+ T) \leq C.$$
Thus, $F_\infty \in C_c^1(\bfR \times [0,\infty))$ and satisfies
\begin{equation}
\label{Weaklimit}
\lim_{t \to \infty} \int_{-\infty}^\infty\int_0^\infty \psi(w,\ell) F(t,w,\ell) d\ell dw = \int_{-\infty}^\infty\int_0^\infty \psi(u,\ell) F_\infty(u,\ell) d\ell du.
\end{equation}
Notice further that due to the compact support and regularity of $F(t)$ and $F_\infty$, we only need $\psi \in L^1_{\mathrm{loc}}(\bfR \times [0,\infty))$ with $\psi$ continuous in $w$ for this equality to hold.

Next, we show that the conservation laws are maintained in the limit, though this will merely follow from \eqref{Weaklimit}.
Indeed, choosing $\psi(w, \ell) = 1$ within \eqref{Weaklimit} and using the time-independence of the total mass, we have
\begin{eqnarray*}
\int_{-\infty}^\infty\int_0^\infty  \finf(\winf, \linf) \ d\linf d\winf 
& = & \lim_{t \to \infty} \int_{-\infty}^\infty\int_0^\infty F(t,w,\ell) d\ell dw \\
& = & \lim_{t \to \infty} \int_0^\infty \int_{-\infty}^\infty\int_0^\infty f(t, r, w, \ell) \ d\ell dw dr\\
& = & \int_0^\infty \int_{-\infty}^\infty\int_0^\infty f_0( r, w, \ell) \ d\ell dw dr\\
& = & \frac{1}{4\pi^2} \mcM.
\end{eqnarray*}
Similarly, for any $\phi \in L^1_{\mathrm{loc}}(0,\infty)$, we choose $\psi(w, \ell) = \phi(\ell)$ within \eqref{Weaklimit}, which yields
\begin{eqnarray*}
\int_{-\infty}^\infty\int_0^\infty  \phi(\ell) \finf(\winf, \linf) \ d\linf d\winf 
& = & \lim_{t \to \infty} \int_{-\infty}^\infty\int_0^\infty \phi(\ell) F(t,w,\ell) dw d\ell\\
& = & \lim_{t \to \infty} \int_0^\infty \int_{-\infty}^\infty\int_0^\infty \phi(\ell) f(t, r, w, \ell) \ drdwd\ell\\
& = & \int_0^\infty \int_{-\infty}^\infty\int_0^\infty \phi(\ell) f_0( r, w, \ell) \ drdwd\ell\\
& = & \mcJ_\phi.
\end{eqnarray*}
In particular, this shows that the total angular momentum is preserved in the limit.

To establish the energy identity, we use the energy conservation of \eqref{VP}, the measure-preserving property, Lemma \ref{L1}, and Lemma \ref{PotDecay} with $p = 2$ to find
\begin{eqnarray*}
\mcEVP & = & \lim_{t \to \infty} \left ( \iint \frac{1}{2} |v |^2 f(t, x,v) \ dv dx + \frac{1}{2} \int |E(t,x)|^2 \ dx \right)\\
& = & \lim_{t \to \infty} \iint \frac{1}{2} |v |^2 f(t, x,v) \ dv dx\\
& = & 2\pi^2\lim_{t \to \infty} \iiint\limits_{\bS(t)} \left (w^2 + \ell r^{-2} \right ) f(t,r,w,\ell) \ d\ell dw dr\\
& = & 2\pi^2\lim_{t \to \infty} \iiint\limits_{\bS(0)} \left (\mcW(t,0,\tilde{r},\tilde{w},\tilde{\ell})^2 + \tilde{\ell} \mcR(t,0,\tilde{r}, \tilde{w}, \tilde{\ell})^{-2} \right )f_0(\tilde{r},\tilde{w},\tilde{\ell}) \ d\tilde{\ell} d\tilde{w} d\tilde{r} \\
& = & 2\pi^2 \iiint\limits_{\bS(0)} \mcW_\infty(0,\tilde{r},\tilde{w},\tilde{\ell})^2 f_0(\tilde{r},\tilde{w},\tilde{\ell}) \ d\tilde{\ell} d\tilde{w} d\tilde{r}.
\end{eqnarray*}
Now, choosing $\psi(w, \ell) = w^2$ within \eqref{Weaklimit} and using the standard change of variables, we have
\begin{eqnarray*}
\int_{-\infty}^\infty\int_0^\infty  w^2 \finf(\winf, \linf) \ d\linf d\winf 
& = & \lim_{t \to \infty} \int_{-\infty}^\infty\int_0^\infty w^2 F(t,w,\ell) d\ell dw\\
& = & \lim_{t \to \infty} \int_0^\infty \int_{-\infty}^\infty\int_0^\infty w^2 f(t, r, w, \ell) \ d\ell dw dr\\
& = & \lim_{t \to \infty}  \iiint\limits_{\bS(0)}  \mcW(t,0,\tilde{r},\tilde{w},\tilde{\ell})^2 f_0(\tilde{r},\tilde{w},\tilde{\ell}) \ d\tilde{\ell} d\tilde{w} d\tilde{r}\\
& = & \iiint\limits_{\bS(0)} \mcW_\infty(0,\tilde{r},\tilde{w},\tilde{\ell})^2 f_0(\tilde{r},\tilde{w},\tilde{\ell}) \ d\tilde{\ell} d\tilde{w} d\tilde{r}.
\end{eqnarray*}
Combining these two equalities then yields the energy conservation law.
To obtain the same result for \eqref{RVP}, we merely repeat these steps to find
$$\mcERVP  = 4\pi^2 \iiint\limits_{\bS(0)} \sqrt{1 + \mcW_\infty(0,\tilde{r},\tilde{w},\tilde{\ell})^2} f_0(\tilde{r},\tilde{w},\tilde{\ell}) \ d\tilde{\ell} d\tilde{w} d\tilde{r}$$
and use $\psi(w,\ell) = \sqrt{1 +  w^2}$ in \eqref{Weaklimit}.
\end{proof}

The final lemma uses the fact that $w$-derivatives of the particle distribution grow more slowly along the associated linear transport characteristics, $r +wt$ or $r + \hat{w}t$, in order to show uniform convergence of the spatial average. A similar idea was recently utilized in \cite{Ionescu} to prove a modified scattering result for small data solutions of \eqref{VPgeneral}, though we will use different analytic tools.

\begin{lemma}
\label{FCconv}
Assume $\bS(0)$ satisfies \eqref{Condition}. Then, the spatial average further satisfies $F(t,w, \ell) \to F_\infty(w, \ell)$ uniformly as $t \to \infty$ with
$$\| F(t) - F_\infty \|_{L^\infty(\bfR \times [0,\infty))} \leq C(1+t)^{-1}(1 + \ln^2(1+t))$$
for all $t \geq 0$.
\end{lemma}

\begin{proof}
We first estimate the size of the $r$-support of $f$ over time. 
%
Begin by taking $t \geq 1$ and letting the points $(r_1,w,\ell), (r_2,w,\ell) \in S(t)$ be given. Integrating the characteristic equations \eqref{charang} and using Lemma \ref{L2} with the boundedness of the spatial support at $t=1$, we find
\begin{eqnarray*}
C & \geq & \left |\mcR(1,t,r_1, w, \ell) - \mcR(1,t,r_2,w,\ell) \right | \\
& \geq & \left | r_1 - r_2 \right | - 2\int_1^t \int_s^t \Vert E(\tau) \Vert_\infty \ d\tau ds\\
& \geq & \left | r_1 - r_2 \right | - C \ln(1+t)
\end{eqnarray*}
for $t \geq 1$.
Rearranging the inequality produces
$$|r_1 - r_2| \leq C(1 + \ln(1+t))$$
for $t \geq 1$.
Thus, for any $t \geq 1$ and $(w,\ell) \in \bfR \times [0,\infty)$, the diameter of the $r$-support  grows like $\mathcal{O}(\ln(t))$, and there are $r_0 \in [0,\infty)$ and $C  > 0$ such that
$$ \{ r \geq 0 : f(t,r,w,\ell) \neq 0\} \subseteq \left \{ r \geq 0 : | r - r_0| \leq C(1+\ln(1+t)) \right \}.$$
Combining this with a constant bound on the support for $t \in [0,1]$, we find
\begin{equation}
\label{Rsupp}
\sup_{(w,\ell) \in \bfR \times [0,\infty)} \left | \{ r \geq 0 : f(t,r,w,\ell) \neq 0\}  \right | \leq C(1+\ln(1+t))
\end{equation}
for $t \geq 0$.

Next, we define for every $t \geq 0$ and $(r,w,\ell) \in [0,\infty) \times \bfR \times [0,\infty)$ the translated distribution function
$$g(t,r, w, \ell) = f(t,r+wt, w, \ell),$$
and note that $g$ satisfies the equation
\begin{equation}
\label{PDEg}
\partial_t g = \left ( \frac{m(t,r+wt)}{(r+wt)^2} + \ell (r+wt)^{-3}\right) (t \partial_r - \partial_w) g
\end{equation}
as $f$ satisfies the Vlasov equation.
Additionally, for $t$ sufficiently large we have $w \geq 0 $ on the support of $f$ and thus 
$$\int_0^\infty g(t,r, w, \ell) \ dr = \int_0^\infty f(t, r+wt, w, \ell) \ dr = \int_0^\infty f(t, r, w, \ell) \ dr = F(t,w, \ell).$$

Now, we estimate
$\partial_w \left [ g(t,r,w,\ell) \right ] = \left ( t \partial_r f + \partial_w f \right )(t,r+wt, w, \ell).$ Applying the Vlasov operator $\mfV$ to the non-translated version of this quantity yields
\begin{eqnarray*}
\mfV \left ( t\partial_r f(t, r, w, \ell) + \partial_w f(t, r, w, \ell)  \right )
& = & - t\partial_{r} \left ( \frac{m(t,r)}{r^2} + \ell r^{-3} \right ) \partial_w f(t, r, w, \ell)\\
& = & -t\left (\rho(t,r) -2\frac{m(t,r)}{r^3} - 3\ell r^{-4} \right) \partial_w f(t, r, w, \ell),
\end{eqnarray*}
and thus inverting via characteristics gives
\begin{eqnarray*}
t\partial_r f(t, r, w, \ell) + \partial_w f(t, r, w, \ell) & = & \partial_w f_0(\mcR(0),\mcW(0), \ell)\\
& \ &  - \int_0^t \left (s \left ( \rho -2m r^{-3} - 3\ell r^{-4} \right )\partial_w f \right ) \biggr |_{(s,\mcR(s), \mcW(s), \ell)} ds.
\end{eqnarray*}
From the estimates of Lemmas \ref{Lrho}, \ref{LAEst}, and \ref{Df}, this implies
\begin{eqnarray*}
\left | t\partial_r f(t, r, w, \ell) + \partial_w f(t, r, w, \ell) \right | 
& \leq & \|\partial_w f_0\|_\infty + C\int_0^t (1+s)^{-2}  \|\partial_w f(s)\|_\infty \ ds \\
& \leq & C + C\int_0^t (1+s)^{-1} \ ds\\
& \leq & C(1 + \ln(1+t))
\end{eqnarray*}
for $(r,w,\ell) \in S(t)$.
Using this estimate, we find
\begin{eqnarray*}
\int_0^\infty  \left | \partial_w g(t,r,w, \ell) \right | \ dr & = & \int_0^\infty \left | t \partial_r f(t,r+wt, w, \ell) + \partial_w f(t,r+wt, w, \ell) \right | \ dr \\
& = & \int_0^\infty  \left | t \partial_r f(t,r, w, \ell) + \partial_w f(t,r, w, \ell) \right | \ dr \\
& \leq & C \sup_{(w,\ell) \in \bfR \times [0,\infty)} \left | \{ r \geq 0 : f(t,r,w,\ell) \neq 0\}  \right | (1 + \ln(1+t)),
\end{eqnarray*}
and thus by \eqref{Rsupp}
\begin{equation}
\label{Dvg}
\int_0^\infty \left | \partial_w g(t,r,w,\ell) \right | \ dr  \leq C(1 + \ln(1+t))^2
\end{equation}
for all $t \geq 0$, $(w, \ell) \in \bfR \times [0,\infty)$.
%
Furthermore, note that for $t$ sufficiently large, Lemma \ref{LAEst} implies
\begin{equation}
\label{momest}
\int_0^\infty (r+wt)^{-p} g(t,r,w,\ell) \ dr
= \int_0^\infty r^{-p} f(t,r,w,\ell) \chfn_{ \{f(t,r,w,\ell) \neq 0\}} \ dr
\leq Ct^{-p} F(t, w, \ell)
\end{equation}
for any $p \geq 0$.

Next, we will focus on showing that the spatial integral of $g$ has a uniform limit as $t \to \infty$.
In particular, upon integrating \eqref{PDEg} in $r$ and integrating by parts, we find
\begin{eqnarray*}
\partial_t \int_0^\infty g(t,r,w,\ell) \ dr & = & \int_0^\infty \left ( \frac{m(t, r+wt)}{(r+wt)^2} + \ell (r+wt)^{-3} \right ) (t \partial_r - \partial_w) g(t,r,w,\ell) \ dr\\
& = & -t  \int_0^\infty \left (\rho(t, r+wt) - 2\frac{m(t,r+wt)}{(r+wt)^3} -3\ell (r+wt)^{-4} \right ) g(t,r,w, \ell) \ dr\\
& \ & \quad  -  \int_0^\infty \left ( \frac{m(t, r+wt)}{(r+wt)^2} + \ell (r+wt)^{-3} \right ) \partial_w g(t,r,w,\ell) \ dr.
\end{eqnarray*}
Thus, because $F(t,w, \ell) = \int_0^\infty g(t,r,w,\ell) \ dr$ for $t$ sufficiently large, we use \eqref{Dvg} and \eqref{momest} with the field decay of Lemma \ref{L2} to find
\begin{equation}
\label{dtf}
\left | \partial_t F(t,w,\ell) \right | \leq C(1+t)^{-2} F(t,w,\ell) + C(1+t)^{-2}(1+ \ln(1+t))^2.
\end{equation}
Because the latter term in this inequality is integrable in time, we find
$$F(t,w,\ell) \leq F(T_3,w,\ell) + \int_{T_3}^t \left | \partial_t F(s,w, \ell) \right | \ ds \leq C + C\int_{T_3}^t (1+s)^{-2} F(s, w, \ell) \ ds$$
where $t \geq T_3$ guarantees $w \geq 0$ on the support of $f$.
After taking the supremum and invoking Gronwall's inequality, this yields
$$\| F(t) \|_\infty \leq C \exp \left (C\int_{T_3}^t (1+s)^{-2} \ ds \right ) \leq C$$
for $t \geq T_3$.
Returning to \eqref{dtf}, we use the bound on $\|F(t)\|_\infty$ to find
$$\left | \partial_t F(t,w,\ell) \right |  \leq C(1+t)^{-2}(1+ \ln(1+t))^2.$$
Of course, this implies that $ \| \partial_t F(t)\|_\infty$ is integrable in time, which guarantees that $F(t)$ is uniformly Cauchy and possesses a limit.
As $F(t,w,\ell)$ is continuous and the limit is uniform, there is $\mfF \in C(\bfR \times [0,\infty))$ such that
$$ \| F(t) - \mfF \|_{\infty} \leq C(1+t)^{-1}(1+ \ln(1+t))^2$$
for $t \geq 0$.
Finally, because $F \to \mfF$ strongly in the $L^\infty$ norm, it converges to this same limit in the weak-$\star$ topology, as well. Then, due to the uniqueness of weak limits, we conclude $\mfF = F_\infty$. 

Turning to solutions of \eqref{RVP}, we merely make a few alterations to the proof. First, a brief calculation shows
\begin{equation}
\label{ddotR}
\ddot{\mcR}(s) = A(s)^{-3} \left ( (1 + \ell \mcR(s)^{-2}) \dot{\mcW}(s) + \ell\mcR(s)^{-3} \mcW(s) A(s)^{-1} \right )
\end{equation}
where
$$A(s) = \sqrt{1 + \mcW(s)^2 + \ell \mcR(s)^{-2}} \geq 1.$$
Using Lemmas \ref{L2}, \ref{L3}, and \ref{LAEst}, we find
$$|\dot{\mcW}(s) | \leq C(1 + s)^{-2}, \qquad
\mcR(s)^{-2} \leq Cs^{-2}, \qquad
\mathrm{and} \qquad
|\mcW(s) | \leq \mfW(s) \leq C.$$
Inserting these estimates into \eqref{ddotR} then provides the bound
$$\left | \ddot{\mcR}(s) \right | \leq Cs^{-2}$$
for $s \geq 1$.

Now, consider $t \geq 1$ and let $(r_1, w, \ell), (r_2, w, \ell) \in S(t)$ be given. Integrating the characteristic equations \eqref{charangrel} yields
$$\mcR(1, t, r_k, w, \ell) = r - \frac{w}{\sqrt{1+w^2 + \ell {r_k}^{-2}}} (t-1) + \int_1^t \int_s^t \ddot{\mcR}(\tau)  \ d\tau ds$$
for $t \geq 1$ and $k=1,2$.
We proceed exactly as for solutions of \eqref{VP} by subtracting these values, with the exception that the middle term in the above expression now depends upon $r_k$ and must be estimated.
Denoting $\tilde{r}_k = \mcR(1,t,r_k,w,\ell)$ and $\tilde{w}_k = \mcW(1, t, r_k, w, \ell)$ for $k=1,2$ and letting $\tilde{\ell} = \ell$, Lemma \ref{LAEst} implies
$$r_1^{-2} + r_2^{-2} = \mcR(t, 1, \tilde{r}_1, \tilde{w}_1, \tilde{\ell})^{-2} + \mcR(t, 1, \tilde{r}_2, \tilde{w}_2, \tilde{\ell})^{-2} \leq Ct^{-2}.$$
Using this bound and the velocity bound from Lemma \ref{L3}, we estimate the difference of the resulting terms by
$$\left | \frac{w}{\sqrt{1 +w^2 + \ell r_1^{-2}}} - \frac{w}{\sqrt{1 +w^2 + \ell r_2^{-2}}} \right | = \frac{\ell |w| \left | r_2^{-2} - r_1^{-2} \right |}{
A_1 A_2 (A_1 + A_2)} \leq C \left (r_1^{-2} + r_2^{-2} \right ) \leq Ct^{-2} $$
where $$A_k(r_k, w, \ell) = \sqrt{1 +w^2 + \ell r_k^{-2}} \geq 1.$$
Hence, the upper bound
$$\left |r_1 - r_2 \right | \leq C\left (1+ t^{-1} +\ln(1+t) \right ) \leq C\left (1+\ln(1+t) \right )$$
for $t \geq 1$ results as for \eqref{VP}.
As before, this implies \eqref{Rsupp}, and the support grows like $\mathcal{O}(\ln(t))$.

Next, we denote the relativistic velocity by
$$\hat{w} = \frac{w}{\sqrt{1 + w^2+ \ell r^{-2}}},$$
and note that
$$ \left | \hat{w} \right | \leq 1, \qquad \left |\frac{\partial \hat{w}}{\partial w} \right | \leq 1, \qquad \mathrm{and} \qquad  \left |\frac{\partial^2 \hat{w}}{\partial w^2} \right | \leq 1.$$
Then, we define for every $t \geq 0$ and $(r,w,\ell) \in [0,\infty) \times \bfR \times [0,\infty)$ the translated distribution function
$$h(t,r, w, \ell) = f(t,r+\hat{w}t, w, \ell)$$
and proceed as for \eqref{VP}.
Note that $h$ satisfies
\begin{equation}
\label{PDEgrel}
\partial_t h = \left ( \frac{m(t,r+\hat{w}t)}{(r+\hat{w}t)^2} + \ell (r+\hat{w}t)^{-3}\right) \left (t \frac{\partial \hat{w}}{\partial w} \partial_r - \partial_w \right) h
\end{equation}
because $f$ satisfies the Vlasov equation.
As before, for $t$ sufficiently large we have $w \geq 0 $ on the support of $f$ and thus
$$\int_0^\infty h(t,r, w, \ell) \ dr = F(t,w, \ell).$$
Furthermore, the $w$-derivative of the translated distribution function now satisfies
$$\partial_w \left [ h(t,r,w,\ell) \right ] = \left ( t \frac{\partial \hat{w}}{\partial w} \partial_r f + \partial_w f \right )(t,r+\hat{w}t, w, \ell).$$

With this, the proof proceeds as for \eqref{VP} with minor exceptions adjusting for the appearance of derivatives of the relativistic velocity. In particular, we must use
$$\left | \partial_r \left ( \frac{\partial \hat{w}}{\partial w} \right ) \right | = \ell r^{-3} \left |\partial_w \left ( \frac{\hat{w}}{1 + w^2 + \ell r^{-2}} \right ) \right | \leq Ct^{-3}$$
for $t >0$ and $(r,w,\ell) \in S(t)$ due to Lemma \ref{LAEst}.
This further implies
$$\left | \mfV \left ( \frac{\partial \hat{w}}{\partial w} \right) \right | = \left | \hat{w} \partial_r \left ( \frac{\partial \hat{w}}{\partial w} \right ) + \left ( \frac{m(t,r)}{r^2} + \ell r^{-3} \right ) \partial_w \left ( \frac{\partial \hat{w}}{\partial w} \right ) \right |\\
\leq Ct^{-2}$$
for $(r,w,\ell) \in S(t)$, so that by Lemmas \ref{Lrho}, \ref{LAEst}, and \ref{Df}, we find
\begin{eqnarray*}
\left | \mfV \left ( t \frac{\partial \hat{w}}{\partial w}\partial_r f(t, r, w, \ell) + \partial_w f(t, r, w, \ell)  \right ) \right |
& = & \left | t \mfV\left ( \frac{\partial \hat{w}}{\partial w} \right ) \partial_r f - t\partial_{r} \left ( \frac{m(t,r)}{r^2} + \ell r^{-3} \right ) \partial_w f \right |\\
& \leq & Ct\left | \mfV \left ( \frac{\partial \hat{w}}{\partial w} \right) \right | + Ct^{-1}\\
& \leq & Ct^{-1}
\end{eqnarray*}
for $t > 0$.
With these alterations, we follow the previous argument and again conclude that $\|F(t) \|_\infty$ is bounded and
$ \| \partial_t F(t)\|_\infty$ is integrable; hence, the result follows for solutions of \eqref{RVP}.
\end{proof}


\section{Proof of Theorem \ref{T1}}
In the final section, we collect a number of estimates from previous lemmas to prove the first of the main results.

\begin{proof}[Proof of Theorem \ref{T1}]

We begin by combining the results of Lemmas \ref{PotDecay} and \ref{L3} to find
$$C (1+t)^{-2p + 3} \leq \int |E(t,x) |^p \ dx \leq C(1+t)^{-2p +3}$$
and
thus
\begin{equation}
\label{Ep}
C(1+t)^{-2 + \frac{3}{p}} \leq \|E(t)\|_p \leq C(1+t)^{-2 +\frac{3}{p}}
\end{equation}
for any $p \in \left ( \frac{3}{2}, \infty \right)$.
Upper and lower bounds for the endpoint case $p =\infty$ are addressed in Lemmas \ref{L2} and \ref{Ebelow} with the bound on $\mfR(t)$ arising from Lemma \ref{L3}, and the matching rate is obtained so that \eqref{Ep} holds for all $p \in \left (\frac{3}{2}, \infty \right].$
Thus, the optimal field decay rate in any suitable $L^p$ norm has been established.
 
Next, we prove the stated decay of the density.
Because the total mass is conserved, we have
$$\| \rho(t) \|_1 = \mcM,$$
and using Lemma \ref{Lrho} and the upper bound on $\mfR(t)$ from Lemma \ref{L3}, we find
\begin{equation}
\label{rhoinfty}
C_1 (1+t)^{-3} \leq \| \rho(t) \|_\infty \leq C_2(1+t)^{-3}
\end{equation}
for some $C_1, C_2 > 0$. Therefore, the upper bound
\begin{equation}
\label{rhoq}
\| \rho(t) \|_q \leq C (1+t)^{-3+ \frac{3}{q}}
\end{equation}
for $q \in [1,\infty]$ merely follows by $L^q$-interpolation between $L^1$ and $L^\infty$. Obtaining the optimal lower bound requires an additional estimate, namely Young's convolution inequality for the electric field (cf. \cite{Glassey}), which yields
$$\| E(t) \|_2 \leq C\Vert \rho(t) \Vert_{\frac{6}{5}}.$$
Using the previously established lower bound for $\| E(t) \|_2$ obtained by choosing $p = 2$ in \eqref{Ep}, this now becomes
\begin{equation}
\label{rho65}
C (1+t)^{-\frac{1}{2}} \leq \Vert \rho(t) \Vert_{\frac{6}{5}}.
\end{equation}
This lower bound is optimal, as we have
$$\Vert \rho(t) \Vert_{\frac{6}{5}} \leq C (1+t)^{-\frac{1}{2}}$$
by selecting $q = \frac{6}{5}$ in \eqref{rhoq} above.
With this, the optimal lower bound on $\|\rho(t)\|_q$ for any $q \in (1,\infty)$ can be obtained.
In particular, if  $q \in \left (1, \frac{6}{5} \right )$ we interpolate the $L^{\frac{6}{5}}$ norm between $L^q$ and $L^\infty$ and use \eqref{rhoinfty} and \eqref{rho65} so that
$$ C(1+t)^{-\frac{1}{2}} \leq \|\rho(t)\|_{\frac{6}{5}} \leq C  \|\rho(t)\|_q^{\frac{5}{6}q}  \|\rho(t)\|_\infty^{1 - \frac{5}{6}q}  \leq C  \|\rho(t)\|_q^{\frac{5}{6}q}  (1+t)^{\frac{5q-6}{2}}.$$
Rearranging the inequality yields
$$C(1+t)^{-3+\frac{3}{q}} \leq \| \rho(t)\|_q$$
as desired.
Similarly, if  $q \in \left (\frac{6}{5}, \infty \right )$ we interpolate the $L^{\frac{6}{5}}$ norm between $L^1$ and $L^q$ and use mass conservation along with \eqref{rho65} so that
$$ C(1+t)^{-\frac{1}{2}} \leq \|\rho(t)\|_{\frac{6}{5}} \leq C  \|\rho(t)\|_q^{\frac{q}{6q-6}}  \|\rho(t)\|_1^{\frac{5q-6}{6q-6}} = C  \|\rho(t)\|_q^{\frac{q}{6q-6}}.$$
Rearranging the inequality again yields the desired lower bound
$$C(1+t)^{-3+\frac{3}{q}} \leq \| \rho(t)\|_q.$$
Combining these estimates with \eqref{rhoq} for $q \in (1,\infty)$, as well as mass conservation and \eqref{rhoinfty} for the boundary cases $q=1$ and $q=\infty$, yields the optimal decay rate for all $L^q$ norms of the charge density for $q \in [1,\infty]$. 
%
%
Finally, the optimal growth rates of the maximal support functions follow immediately from Lemma \ref{L3}.
\end{proof}

\end{document}